\documentclass{article}

 \usepackage{amsmath}
\usepackage{amsthm}
\usepackage{amssymb}

\allowdisplaybreaks[3]

\theoremstyle{plain}
\newtheorem{theorem}{THEOREM}[section]
\newtheorem{proposition}[theorem]{PROPOSITION}
\newtheorem{lemma}[theorem]{LEMMA}

\theoremstyle{definition}
\newtheorem{definition}[theorem]{DEFINITION}

\theoremstyle{remark}
\newtheorem{remark}[theorem]{REMARK}

\numberwithin{equation}{section}

\newcommand{\e}{\mathrm e}
\newcommand{\mySin}{\mathrm s}
\newcommand{\myCos}{\mathrm c}
\newcommand{\F}{\underline{F}}
\newcommand{\B}{\tilde{B}}

\newcommand{\HH}{\tilde{\xi}}

\newcommand{\abs}[1]{\lvert #1 \rvert }
\newcommand{\Abs}[1]{\Bigl\lvert #1 \Bigr\rvert }

\newcommand{\frp}[1]{\{ #1\}}
\newcommand{\iinp}[1]{[[#1]]}
\newcommand{\ffrp}[1]{\langle \langle #1 \rangle \rangle}

\title
{An Elliptic Analogue Of Generalized Cotangent Dirichlet Series And Its Transformation Formulae At Some Integer Arguments}
\author{MACHIDE, Tomoya}

\begin{document}

\maketitle

\begin{abstract}
	B.C. Berndt evaluated special values of the $cotangent$ $Dirichlet$ $series$.
	T. Arakawa studied a generalization of the series, or $generalized$ $cotangent$ $Dirichlet$ $series$,
		and gave its transformation formulae.
		
	In this paper, we establish an elliptic analogue of the generalized cotangent Dirichlet series 
		and give its transformation formulae at some integer arguments.
	As a corollary, we obtain the transformation formulae of the generalized cotangent Dirichlet series at some integer arguments
		which are the part of Arakawa's transformation formulae.
	Those transformation formulae give the special values of the cotangent Dirichlet series evaluated by B.C. Berndt.
\end{abstract}

\section{Introduction}\label{Sect_Introduction}
The $cotangent\  Dirichlet\  series$
	\begin{equation}
	\xi ( s, \alpha ) := \sum_{n=1}^{\infty} \frac{\cot \pi n \alpha}{n^s}
	\end{equation} 
	for an irrational real algebraic number $\alpha$ over $\mathbb{Q}$
	has been studied by B.C. Berndt \cite{berndt1}.
He showed that $\xi ( s, \alpha )$ is absolutely convergent if $\mathrm{Re}(s)$ is larger than the degree of $\alpha$,
	and evaluated the special values of $\xi(s, \alpha)$ for positive odd integers $s$ and real quadratic numbers $\alpha$.
After him, T. Arakawa improved the bounds of convergence from the degree of $\alpha$ to $1$
	by use of the Thue-Siegel-Roth theorem.
\begin{theorem}\label{1.Berndt_Theorem}
$\mathrm{(B.C. Berndt}$ \cite[Theorem 5.1, 5.2]{berndt1} with  $\mathrm{T. Arakawa}$ \cite[Lemma 1]{arakawa1}$\mathrm{)}$\\
$\mathrm{(i)}$
If $s$ is a real number with $\mathrm{Re}(s) > 1$ and $\alpha$ is an irrational real algebraic number, then 
	$\xi (s, \alpha)$ is absolutely convergent.
\\
$\mathrm{(ii)}$
Let $\alpha = (a + b \sqrt{c})/2$  and $\epsilon = \pm 1$, 
	where 
	$a, b$ and $c$ are integers with $c > 0$ and
	$a^2 -cb^2 = 4 \epsilon$. If $l$ is an integer with $l > 1$, then
	\begin{equation}\label{1.Berndt_formula}
		(1 - \epsilon \alpha^{2l-2})\xi(2l-1, \alpha)
		=
		\frac{(-1)^{l-1}  (2\pi)^{2l-1}}{(2l)!}
			\sum_{k=0}^{l} \binom{2l}{2k}  \alpha^{2k-1} B_{2k}B_{2l-2k}.
	\end{equation}
Here $B_k$ are the k-th Bernoulli numbers defined by
	$\sum\limits_{m=0}^{\infty} (B_m/m!)x^{m} = x / (\e^x -1)$.
\end{theorem}
For real numbers $s$ and $x$,
	let $\e(s)$ denote $\e^{2 \pi i s}$,
	and
	$\langle x \rangle$ (resp. $\{x \}$) the real number which satisfies
	$0 < \langle x \rangle \leq 1$ and $x - \langle x \rangle \in \mathbb{Z}$
	(resp. $0 \leq \{x \} < 1$ and $x - \{ x \} \in \mathbb{Z}$).
Furthermore let $\chi (x)$ be the characteristic function of integers, i.e., 
	\begin{equation*}
		\chi (x) := 
		 	\begin{cases}
				 1 & \mathrm{if} \quad  x \in \mathbb{Z}, \\
				 0 & \mathrm{if} \quad  x \notin \mathbb{Z},
		 	\end{cases}
	\end{equation*}
	and $\zeta(s, x)$ the Hurwitz zeta function defined by
	\begin{equation*}
		\zeta(s, x) := \sum_{n = 0}^{\infty} \frac{1}{(n + x)^s}.
	\end{equation*}
For any pair $\vec{\omega} = (\omega_1, \omega_2)$ of positive real numbers 
	and for complex numbers $z, t \in \mathbb{C}$,
	we set
	\begin{equation*}
		G_2 (z, \vec{\omega}; t) := 
			\frac{\exp(-zt)}{(1-\exp(-\omega_1t ))(1-\exp(-\omega_2 t))}.
	\end{equation*} 

T. Arakawa \cite{arakawa1} established transformation formulae for the infinite series 
	\begin{equation} \label{1.definition_H}
		H (\alpha,s,x,y) 
		:=
		\sum_{n=1}^{\infty} \frac{\e(ny)}{n^{1-s}} \frac{\e (n \langle x \rangle \alpha)}{1-\e(n \alpha)}
		+
		\e(s/2)
		\sum_{n=1}^{\infty} \frac{\e(-ny)}{n^{1-s}} \frac{\e(n \langle -x \rangle  \alpha)}{1-\e(n \alpha)}.
	\end{equation}
\begin{theorem}\label{1.Arakawa_Theorem}
$\mathrm{(T. Arakawa}$ \cite[Theorem 1]{arakawa1}$\mathrm{)}$
	Let $\alpha$ be any irrational real algebraic number, 
		and let $V = \begin{pmatrix} a & b \\ c & d \end{pmatrix}  \in \mathrm{SL}_2 (\mathbb{Z})$ 
		with $c > 0$ and $c \alpha + d > 0$.
Put $\beta = c \alpha + d$,
	and
	set $p' = pa + qc$, $q' = pb + qd$,
		and $\rho = \{ q' \} c - \{p' \} d$ for any pair $(p, q)$ of real numbers.
	If $\mathrm{Re}(s) < 0$, then
		\begin{eqnarray} 
			& & \beta^{-s} H(V\alpha, s, p, q) \nonumber \\
			& = &  H(\alpha, s, p', q')
				- \chi(p) (2 \pi)^{-s} \e(s/4) \beta^{-s} \Gamma(s) \bigl(\zeta(s, \langle q \rangle)
				+\e(s/2) \zeta(s, \langle - q \rangle) \bigr) \label{1.Arakawa_Equation} \nonumber \\
			& & \quad
				- \chi(p') (2 \pi)^{-s} \e(-s/4) \beta^{-s} \Gamma(s) \bigl(\zeta(s, \langle -q' \rangle)
				+\e(s/2) \zeta(s, \langle q' \rangle) \bigr)\nonumber \\
			& & \quad
				+(2 \pi)^{-s}  \e(-s/4) L(\alpha, s, p', q', c, d). 
		\end{eqnarray}
	Here $\Gamma(s)$ is the gamma function and 
		\begin{equation*}
			L(\alpha, s, p', q', c, d)
			=
			- \sum_{j=1}^{c}
				\int\limits_{I(\lambda, \infty)}
					t^{s-1} G_2(1-\{ \frac{jd + \rho}{c} \} +  \beta \frac{j -\{ p' \}}{c}, (1, \beta); t) dt
		\end{equation*}
		where $I(\lambda, \infty)$ is the integral path consisting of the oriented half line 
		$(+\infty, \lambda)$, a counterclockwise circle of radius $\lambda$ around the origin,
		and the oriented half line $(\lambda, +\infty)$.
\end{theorem}
\begin{remark}
	Let $u, \omega$ be positive real numbers, and $s$ a complex number.
	The Barnes zeta function
	\begin{equation*}
		\zeta_2(s, \omega, u)
		:=
		\sum_{m,n=0}^{\infty}
			\frac{1}{(u + m +n\omega)^s}
	\end{equation*}
	has been intensively studied by Barnes \cite{barnes1}.
	$L(\alpha, s, p', q', c, d)$  is rewritten in terms of Barnes zeta function as follows (see \cite[(1.18)]{arakawa2}) .
	\begin{equation*}
		L(\alpha, s, p', q', c, d)
		=
		- \Gamma(s) (\e (s) - 1)\sum_{j=1}^{c}
			\zeta_2 (s, \beta, 1-\{ \frac{jd + \rho}{c} \} +  \beta \frac{j -\{ p' \} }{c} ).
	\end{equation*}
	
\end{remark}

We define the function $\HH (s,\alpha,x,y)$ by
	\begin{equation}\label{1.definition_generalized_H}
		\HH (s,\alpha,x,y) := - H (\alpha,1-s,- y,x)
	\end{equation}
	which is modified the arguments $\alpha, s, x, y$ of $H (\alpha,s,x,y)$.
The purpose of the modification is to express our results easily
	(for example, see (\ref{Sect1.mod}) in which the argument $s$ is adjusted to that of $\xi (s, \alpha)$,
	and
	(\ref{4.theorem1a}) in which transformation formulae on the modular group is effectively described).
The cotangent Dirichlet series $\xi ( s, \alpha )$ is expressed by 
	the function $\HH (s,\alpha,x,y)$ with $(x,y)=(0,1)$:
\begin{equation}\label{Sect1.mod}
	\xi (s, \alpha)
	=
	-2i \big{(} \frac{1}{1+\e(s/2)} \HH (s, \alpha, 0, 1) + \frac{1}{2} \zeta (s) \big{)},
\end{equation}
	where
	$\zeta (s) := \sum\limits_{n=1}^{\infty} 1/n^s$ is the Riemann zeta function. 
Therefor we call $\HH (s,\alpha,x,y)$ $generalized$ $cotangent$ $seires$.

The first aim of this paper is to establish elliptic analogue 
	 $\HH (s, \alpha, x', x, y', y; \tau)$  to $\HH (s, \alpha,x,y)$
	 which we call 
	 \emph{elliptic generalized cotangent Dirichlet series} (see Definition \ref{2.def_H}).
The second is 
	to give transformation formulae of the elliptic generalized cotangent Dirichlet series
	at some integer arguments, or at $s \in \mathbb{Z}$ with $s > 2$ (see Theorem \ref{4.Theorem1}).
It should be noted that the argument $s$ is allowed any complex number $s$ with $\mathrm{Re}(s) > 1$
	in Arakawa's transformation formulae of $\HH (s, \alpha,x,y)$.
In particular, if $\alpha$ is a real quadratic number, 
	then $\xi (s, \alpha)$ is meromorphic with respect to $s$
	and its transformation formulae hold at all complex numbers $s$.  
(We note that 
	T. Arakawa \cite{arakawa2} gave certain relations 
	between residues at poles of $\xi(s, \alpha)$ and special values of partial zeta-function and Hecke L-functions.)
On the other hand, in our transformation formulae, the argument $s$  is restricted positive integers at least $3$.
The reason of the restriction is that it is difficult to apply Arakawa's way of establishing transformation formulae of  $\HH (s,\alpha,x,y)$ to our case.
For applying, we need  an integral representation of $\HH (s, \alpha, x', x, y', y; \tau)$ and an elliptic analogue of Barnes zeta function.
However,
	we can give Berndt's Theorem \ref{1.Berndt_Theorem} (ii) from our transformation formulae,
	because
	it is derived from Arakawa's transformation formulae in the case that $s$ is an odd integer with $s \geq 3$
	which follow from our transformation formulae of $\HH (s, \alpha, x', x, y', y; \tau)$ by $\tau \rightarrow i \infty$.

The paper is organized as follows:
In Section \ref{Sect_H}, we define the elliptic generalized cotangent Dirichlet seires $\HH (s, \alpha, x', x, y', y; \tau)$ 
	and show that it is absolutely convergent if $\mathrm{Re}(s)  > 2$.
In Section \ref{Sect_EDRSum},	 we prepare $Elliptic\  Dedekind$-$Rademacher\  sums$
	for next Section \ref{Sect_Transformation}
	in which we establish transformation formulae of $\HH (s, \alpha, x', x, y', y; \tau)$
	with integers $s > 2$.
In Section \ref{Sect_BerndtArakawaResults}, 
	from our transformation formulae, 
	we derive (\ref{1.Arakawa_Equation}) with integers $s < -1$
	and Theorem \ref{1.Berndt_Theorem} (ii).
Section \ref{Sect_Lemma1} and \ref{Sect_Lemma2} devote the proofs of Lemma \ref{4.Lemma1} and \ref{4.Lemma2} 
	in Section \ref{Sect_Transformation} respectively.
   
Throughout the paper, let 
	$\alpha$ be an irrational algebraic number over $\mathbb{Q}$,
	$s$ a complex number, 
	$y', y, x', x$ real numbers,
	$\tau$ a complex number with positive imaginary part,
	and $\mathrm{SL}_{2} (\mathbb{Z})$ the modular group.
	If $\mathbb{A}$ is a ring, 
		$\mathrm{M}_2 (\mathbb{A})$ means the set of two by two matrixies whose entries are in $\mathbb{A}$.
We use the following notions:  
	$\e(x) := \e ^{2\pi i ix}$, 
	$V z := (az + b)/(cz + d)$ 
	and
	$j(V;z) := cz + d$
	where $z \in \mathbb{C}$ and 
	$V = \bigl( \begin{smallmatrix} a & b \\ c & d  \end{smallmatrix} \bigr)
		\in \mathrm{SL}_2 (\mathbb{Z})$.

\section{Elliptic generalized cotangent Dirichlet series} \label{Sect_H}
In this section, we fulfill the first aim of this paper, or define the elliptic generalized cotangent Dirichlet series  $\HH (s, \alpha, x', x, y', y; \tau)$
	as analogue to generalized cotangent Dirichlet series $\HH (s, \alpha, y, x; \tau)$.
For the aim, we introduce the function $\F (x', x;X;\tau)$ which is analogue to 
	$\dfrac{\e(n\langle x \rangle \alpha)}{1-\e(n \alpha)}$
	used in $\HH (s,\alpha,x,y) $.

Let $q=\e(\tau)$.
The function $\F (x', x;X;\tau)$ is built by Jacobi's theta function
	\begin{align}\label{2.theta}
 		\theta (x ;\tau) 
 		&:= 
 		\sum_{m \in \mathbb Z} \e (\frac 12(m+\frac 12)^2\tau +(m+\frac 12)(x+\frac 12))
 		\\
 		& =
 		i q^{1/8} (\e(\frac{x}{2}) - \e(-\frac{x}{2}))
 			\prod_{m=1}^{\infty}
		(1-\e(-x)q^m)(1-\e(x)q^m)(1-q^m)  \nonumber
	\end{align}
	which is an odd and quasi periodic entire function:
\begin{equation}\label{2.theta_quasi}
\begin{split}
	\theta (-x;\tau) &= -\theta (x;\tau), \\
 	\theta (x+1 ;\tau) &=  -\theta (x ;\tau) , \\ 
	\theta (x+\tau ;\tau) &=  -\e(-\frac{\tau}{2} -x)  \theta (x ;\tau). 
\end{split}
\end{equation}
Let $\theta '(x ;\tau)$ denote the derivative of $\theta (x ;\tau)$ with respect to $x$.
For any two tuple of real numbers $\vec{x} = (x',x) \in \mathbb{R}^2 \setminus \mathbb{Z}^2$,
	the function $\F (\vec{x};X;\tau)$ is defined by
	\begin{equation}\label{2.definition_K}
		\F (\vec{x};X;\tau) :=
		\e (x X)
		\frac{\theta '(0;\tau) \theta(-x' +x\tau +X;\tau)}{\theta(-x' + x \tau;\tau) \theta(X;\tau)}
	\end{equation}
	which is Kronecker double series introduced in \cite{weil} and essentially a meromorphic Jacobi form studied in \cite[Section 3]{zagier1}.
Some fundamental properties of $\F (\vec{x};X;\tau)$ are following: 
As a function with respect to $X$,
	it is meromorphic with only simple poles on the lattice 
 	$\mathbb Z +\tau \mathbb Z$.
By (\ref{2.theta_quasi}) it has properties
	\begin{align}
		\F (\vec{x};X +1;\tau) &= \e (x) \F (\vec{x};X;\tau),&
		\F (\vec{x};X +\tau;\tau) &= \e (x') \F (\vec{x};X;\tau),& \label{2.key_formula1} \\
		\F (-\vec{x};-X;\tau) &= - \F (\vec{x};X;\tau),&
		\F (\vec{x} +\vec{a};X ;\tau) &= \F (\vec{x};X;\tau)& \label{2.key_formula2}
	\end{align}
	where $\vec{a} = (a', a) \in \mathbb{Z}^2$.

If $y \notin \mathbb{Z}$, we find from (\ref{1.definition_H}) and (\ref{1.definition_generalized_H})
	that the infinite series $\HH (s, \alpha, y, x)$ is express as
	\begin{equation} \label{2.infiniteExpression_H}
		\HH (s, \alpha, y, x)
		=
		\sum_{m \in \mathbb{Z} \atop (m \neq 0)}
			\frac{\e(mx)}{m^s} \frac{\e(\alpha m \langle -y \rangle)}{\e(\alpha m) -1}.
	\end{equation}
By considering $\F (\vec{x};X;\tau)$ to be an elliptic generalization to $\e(\langle x \rangle X)/(\e(X) -1)$, 
	we define the elliptic analogue $\HH (s, \alpha, x', x, y', y; \tau)$
	of $\HH (s, \alpha, y, x; \tau)$ as follows:
\begin{definition}\label{2.def_H}
	Let $\alpha$ be an irrational real algebraic number
		and $s$ a complex number with $\mathrm{Re}(s) > 2$.
	For four real numbers $x', x, y', y$ with $(y', y) \notin \mathbb{Z}^2$,
		the elliptic generalized cotangent Dirichlet series $\HH (s, \alpha, x', x, y', y; \tau)$
		is defined by
		\begin{equation} \label{3.infinite_series}
			\HH (s, \alpha, x', x, y', y; \tau)
			:=
			\sideset{}{^\prime}\sum_{m', m}
			\frac{\e (m'x' +mx)}{(\tau m' +m)^{s}}
			\F (-\vec{y}; \alpha (\tau m' +m); \tau)
		\end{equation}
		where the summation 
		ranges over all elements in $\mathbb{Z}^2$ except $(0,0)$.
\end{definition}
We prove the convergency of $\HH (s, \alpha, x', x, y', y; \tau)$.
\begin{lemma}\label{3.Lemma1}
	For any complex number $s$ with $\mathrm{Re}(s) > 2$,
		$\HH (s, \alpha, x', x, y', y; \tau)$
		absolutely converges.
\end{lemma}
\begin{proof}
	Let $LHS$ mean 
		the absolute value of $\HH (s, \alpha, x', x, y', y; \tau)$.
	Set $\vec{x} = (x', x)$ and $\vec{y} = (y', y)$.
	Because the function $\F (x', x' ;X;\tau)$ with respect to $X$ 
 		is meromorphic with only simple poles on the lattice 
 		$\mathbb Z +\tau \mathbb Z$,
 		there is a positive real number $C = C (x', x, \tau)$ depending on $x', x, \tau$
 		which satisfy the following property:
 	If $X= \tau \xi' +\xi$ is a complex number with 
 		$-1/2 \leq \xi', \xi \leq 1/2$, then
		\begin{equation*}
 			\lvert X \F (\vec{x}; X; \tau) \rvert  \leq C .
		\end{equation*} 
 	For any real number $r$, 
 		let $[[ r ]]$ and
 		$\langle \langle r \rangle \rangle$ mean the integer and the real number satisfying
 		$-1/2 < \langle \langle r \rangle \rangle \leq 1/2$ and $r = [[r]] +  \langle \langle r \rangle \rangle $.
 	By (\ref{2.key_formula1}),
		\begin{equation*}
 			\F (\vec{x}; \alpha (\tau m' +m); \tau) 
			=
			\e(x' [[\alpha m' ]] +x [[\alpha m]])
			\F (\vec{x}; \tau \langle \langle \alpha m'\rangle \rangle +\langle \langle \alpha m \rangle \rangle; \tau) ,
		\end{equation*}
		thus we have
		\begin{equation*}
			LHS \leq C \sideset{}{^\prime}\sum\limits_{m', m}
 			\dfrac{1}{{\lvert \tau m' +m \rvert }^{s}}
			\frac{1}{\lvert \tau \langle \langle \alpha m'\rangle \rangle +\langle \langle \alpha m \rangle \rangle  \rvert}.
		\end{equation*}
	If $r'$ and $r$ are real numbers, it holds that 
 		\begin{equation}\label{3.lemma1aa}
 			\begin{split}
  				&  \abs{\tau r' +r} \geq  \abs{\mathrm{Im}(\tau)} \times \bigl(  \max \{ \abs{r'},  \abs{r} \}\bigr), \\
  				&  {\lvert \tau r' +r \rvert}^2  \geq 2 (\lvert \tau \rvert - \lvert \mathrm{Re}\ \tau \rvert) \lvert r' r \rvert.
 			\end{split}
 	 	\end{equation}
 	The first inequality is derived from 
		\begin{equation*}
			\abs{\tau r' +r}^2 
				=
				\abs{\tau}^2 r'^2 + (\tau + \overline{\tau}) r' r + r^2 
				=
				{(\mathrm{Re}(\tau) r'  + r)^2 + (\mathrm{Im}(\tau) r' })^2 \geq  (\mathrm{Im}(\tau) r')^2
		\end{equation*}
		and $\abs{\tau r' +r}^2 = \abs{\tau}^2 \abs{ r' + (1/ \tau)r}^2$.
	Since 
		\begin{equation*}
			0 \leq (\abs{\tau} r' \pm r)^2 =  \abs{\tau}^2 r'^2 + r^2 \pm 2 \abs{\tau} r'r, 
		\end{equation*}
	one has 
		\begin{equation*}
			\abs{\tau r' \pm r}^2 - (\tau + \overline{\tau}) r'r 
			= 
			\abs{\tau}^2 r'^2 + r^2  
			= 
			\mp 2 \abs{\tau} r'r
			\geq 
			2 \abs{\tau} \abs{r' r}.
		\end{equation*}	
	Thus the second is derived from
		\begin{equation*}
			\abs{\tau r' \pm r}^2  \geq 2 \abs{\tau} \abs{r' r} + (\tau + \overline{\tau}) r'r \geq (2 \abs{\tau} - \abs{\tau + \overline{\tau}} )\abs{r' r}
			=
			2(\abs{\tau}-\abs{\mathrm{Re}\ \tau})\abs{r' r}.
		\end{equation*}
	Since $\lvert \tau \rvert - \lvert \mathrm{Re}\ \tau \rvert$ is positive because of $\mathrm{Im}(\tau) > 0$,
	there is a positive real number $D = D(\tau)$ depending on $\tau$ such that the following inequality holds:
	\begin{eqnarray}\label{2.proofOfLemma_11}
		& &
		LHS 
		\\
		&\leq&
		C D 
		\Biggl[
			\sideset{}{^\prime}\sum_{m'} \sideset{}{^\prime}\sum_{m}
				\frac{1}{{\lvert m' \rvert}^{s/2} {\lvert m \rvert}^{s/2} 
				{\lvert \langle \langle \alpha m' \rangle \rangle \rvert}^{1/2}
		 		{\lvert \langle \langle \alpha m \rangle \rangle \rvert}^{1/2}} 
			\nonumber \\
			& & \hspace{50pt}
			+ 
			\sideset{}{^\prime}\sum_{m'} \frac{1}{{\lvert m' \rvert}^{s} 
				\lvert \langle \langle \alpha m' \rangle \rangle \rvert}
			+ 
			\sideset{}{^\prime}\sum_{m} \frac{1}{{\lvert m \rvert}^{s} 
				\lvert \langle \langle \alpha m \rangle \rangle \rvert}
		\Biggr]
		\nonumber \\
		& \leq &
		C D 
		\Biggl[
			\Bigl(
				\sideset{}{^\prime}\sum_{m'} 
					\frac{1}
					{{\lvert m' \rvert}^{s/2} {\lvert \langle \langle \alpha m' \rangle \rangle \rvert}^{1/2}} 
			\Bigr)
			\Bigl(
				\sideset{}{^\prime}\sum_{m}
					\frac{1}{{\lvert m \rvert}^{s/2} {\lvert \langle \langle \alpha m \rangle \rangle \rvert}^{1/2}} 
			\Bigr)
			\nonumber \\
			& & \hspace{50pt}
			+ 
			\sideset{}{^\prime}\sum_{m'} \frac{1}{{\lvert m' \rvert}^{s} 
				\lvert \langle \langle \alpha m' \rangle \rangle \rvert}
			+ 
			\sideset{}{^\prime}\sum_{m} \frac{1}{{\lvert m \rvert}^{s} 
				\lvert \langle \langle \alpha m \rangle \rangle \rvert}
		\Biggr]	
		\nonumber,
	\end{eqnarray}
		where the summation $\sum_{m'}'$ and $\sum_{m}'$ 
		range over all integers except zero respectively. 
	T. Arakawa showed that the series 
		$\sum\limits_{m=1}^{\infty} \dfrac{1}{m^{1+\epsilon} \lvert \langle \langle \alpha m \rangle \rangle\rvert}$
 		converges if $\epsilon >0$
		in the proof of \cite[Lemma 1]{arakawa1}.
	Because $1/\sqrt{x} < 1/x $ if $0 < x < 1$,
		the series 
		$\sum\limits_{m=1}^{\infty} \dfrac{1}{m^{1+\epsilon} {\lvert \langle \langle \alpha m \rangle \rangle\rvert}^{1/2}}$
		does too.
	Therefore the right hand side of (\ref{2.proofOfLemma_11}) converges
		which complete the proof.
\end{proof}

\section{Elliptic Dedekind-Rademacher sums } \label{Sect_EDRSum}
We introduce $Elliptic$ $Dedekind$-$Rademacher$
	$sums\ S_{m,n}(r, x', x, y', y; \tau) $
	for describing transformation formulae of $\HH (s, \alpha, x', x, y', y; \tau)$
	which will be established next section.

These sums are built up by $elliptic\ Bernoulli\ functions\ B_m(x', x;\tau) $ which are defined by 
	use of the generating function $\F (x' ,x ;X;\tau)$:
\begin{equation}\label{3.definition_B}
	\F (x' ,x ;X;\tau) =
	\sum_{m=0}^{\infty}
		\frac{B_m(x' ,x ;\tau)}{m!} {(2\pi i)}^{m} X^{m-1}.
\end{equation}
Elliptic Bernoulli functions are a part of Kronecker's double series (see \cite{weil}).
They have the following explicit expressions (for example, see \cite{machide})
	\begin{multline} \label{Sect3.FouExp}
	 	B_m(\vec{x};\tau) 
	 	=
		m \biggl( 
			\sum_{j=1}^{\infty} (x-j)^{m-1} \frac{\e(-x\tau) q^j}{\e (-x') -\e (-x\tau)q^j} 
		 	 \\
			-\sum_{j=1}^{\infty} (x+j)^{m-1}\frac{\e(x\tau)q^j}{\e (x')-\e(x\tau)q^j} 
			+
			x^{m-1} \frac{\e (-x'+x\tau)}{\e (-x'+x\tau) -1}
		\biggr)  
		+
		B_m(x),
	\end{multline}
	where $q = \e(\tau)$ and $B_m(x)$ are Bernoulli polynomials defined by 
	$\dfrac{\e (x X)}{\e (X)-1} =
		\sum\limits_{m=0}^{\infty}
		\dfrac{B_m(x)}{m!} (2 \pi i X)^{m-1}$. 
For $\vec{x} = (x', x) \in \mathbb{R}^2 \setminus \mathbb{Z}^2$ and $\vec{a} = (a', a)  \in \mathbb{Z}^2$,
	they satisfy by (\ref{2.key_formula2}) that
	\begin{equation}\label{2.key_formula3}
		B_m (-\vec{x};\tau) = (-1)^m B_m (\vec{x};\tau), 
		\quad
		B_m (\vec{x} +\vec{a} ;\tau) = B_m (\vec{x};\tau).
	\end{equation}
We note that,
	for $ \vec{x} \in \mathbb{R}^2$, 
	$B_m(\vec{x};\tau)$
	is discontinuous if $m = 1,2$ and $ \vec{x} \in \mathbb{Z}^2$,  
	and continuous otherwise because of (\ref{Sect3.FouExp}).

For defining the elliptic Dedekind-Rademacher sums $S_{m,n}(r, x', x, y', y; \tau)$,
	we prepare the $numerator$ and $denominator$ maps $n$ and $d$ defined by
	\begin{equation} \label{3.Def_ndNumber}
 		\begin{matrix}
			& n: \mathbb{Q} \to \mathbb{Z},  \quad & r \mapsto p,\\
			& d: \mathbb{Q} \to \mathbb{Z}_{>0},  \quad &  r \mapsto q
 		\end{matrix}
	\end{equation} 
	where $p, q$ are a unique pare of integers such that  
	$\gcd (p, q) =1, q \geq 1$ and $r = p/q$.

\begin{definition}\label{3.definition_ERDSUM}
	For nonnegative integers $m, n$, any rational number $r \neq 0$ and real numbers $x', x, y', y$, 
		we define elliptic Dedekind-Rademacher sums by
		\begin{equation}\label{2.elliptic_Dedekind-Rademacher_sum}
			S_{m,n}(r, x', x, y', y; \tau)
			:=
			\frac{1}{d(r)}
			\sum_{j', j (   d(r) )} 
			B_m \Bigl(\frac{\vec{j} + \vec{y}}{d(r)}; \tau \Bigr)
			B_n \Bigl(n(r) \frac{\vec{j} + \vec{y}}{d(r)} - \vec{x}; \tau \Bigr),
		\end{equation} 
		where the summation ranges over $(j',j) \in {(\mathbb{Z}/d(r)  \mathbb{Z})}^2$
		and  $\vec{j}, \vec{x}, \vec{y}$ denotes the vectors $(j',j), (x',x), (y',y)$ respectively.
	Because of discontinuity of $B_1 (\vec{z}; \tau )$ and $B_2 (\vec{z}; \tau )$ at $ \vec{z} \in \mathbb{Z}^2$,
		we suppose that 
		$\vec{y}, n(r) \vec{y} - d(r) \vec{x} \notin \mathbb{Z}^2$
		if $m, n \in \{ 1,2 \}$.
\end{definition}
We note 
	the relation between these sums and 
	the sums $S_{m,n}^{\tau} \bigl(
	\begin{smallmatrix}
        		(a',a)  &  (b',b)  & (c',c)     \\
        		(x',x)  &  (y',y)  & (z',z)    \\
	\end{smallmatrix} \bigr)$
	defined in \cite{machide}:
	\begin{equation}\label{2.relation}
		S_{m,n}(r, x', x, y', y; \tau)
		=
		S_{m,n}^{\tau}
		\begin{pmatrix}
        			(1,1)  &  (n(r),n(r))  & (d(r),d(r))     \\
        			(0,0)  &  (x',x) & (y', y)     \\
		\end{pmatrix}.
	\end{equation}

\section{Transformation formulae of elliptic generalized cotangent Dirichlet series} \label{Sect_Transformation}
In this section, we give transformation formulae of 
	elliptic generalized cotangent Dirichlet series $\HH (s, \alpha, x', x, y', y; \tau)$
	at integers $s > 2$,
	which is  the second aim of this paper.

Firstly we modify and prepare some notations for describing the transformation formulae simply.
The four arguments $(x', x, y', y)$ used in 
	$\HH (s, \alpha, x', x, y', y; \tau)$ and $S_{m,n}(r, x', x, y', y; \tau)$
	are rewritten in terms of a matrix element of $\mathrm{M}_{2} (\mathbb{R})$, or
	$M=
	\bigl(
	\begin{smallmatrix} 
		x' & x  \\ 
		y' & y
	\end{smallmatrix}
	\bigr). $
Thus 
	\begin{equation}\label{4.equa_Modify}
		\begin{split}
			&
			\HH (s,\alpha, M; \tau)
			:=
			\HH (s, \alpha, x', x, y', y; \tau), \\
			&
			S_{m,n}(r, M; \tau)
			:=
			S_{m,n}(r, x', x, y', y; \tau).
		\end{split}
	\end{equation}
Hereafter we sometimes write the matrix $M$ as 
	$\bigl( \begin{smallmatrix} \vec{x}  \\ \vec{y} \end{smallmatrix} \bigr) $
	where $\vec{x} = (x', x)$ and $\vec{y} = (y', y)$.
If $m \in \mathbb{Z}$, $z \in \mathbb{C}$, 
	$M = 
		\bigl( \begin{smallmatrix} \vec{x}  \\  \vec{y} \end{smallmatrix} \bigr) 
		\in \mathrm{M}_{2} (\mathbb{R})$
	and
	$V =
		\bigl(
			\begin{smallmatrix} a & b  \\  c & d \end{smallmatrix}
		\bigr)
		\in \mathrm{SL}_{2} (\mathbb{Z})$
	with
	$\vec{y}, c \vec{x}+d \vec{y} \notin \mathbb{Z}^2$,
	then we define polynomials $R_V (l, z, M ; \tau)$ of $z$ by 
	\begin{multline}
		R_V (l, z, M; \tau)\\
		:=
		 \begin{cases} \displaystyle{}
		 	\frac{(2 \pi i)^{l+1}}{(l+1)!}
			\sum_{k=-1}^{l}
			\binom{l+1}{k+1} 
			(-j(V;z))^k
				S_{k+1,l-k} 
				\biggl( \frac{d}{c},
					\bigl(\begin{smallmatrix} -\vec{x}  \\ \vec{y}\end{smallmatrix} \bigr) ; \tau \biggr) 
				& 
				( c > 0), \\
			0 & (c =0), \\
			R_{-V} (l, z, M; \tau) & (c < 0)
		\end{cases}
	\end{multline}
	where $j(V;z) := c z + d$ is the automorphic factor on $\mathrm{SL}_{2} (\mathbb{Z})$.   
We note that these polynomials derive period polynomials as follows (See \cite{fukuhara1} about period polynomials): 
If $V =
	\bigl(
	\begin{smallmatrix} 
		0 & -1  \\ 
		1 & 0
	\end{smallmatrix}
	\bigr),
	$
	then
	\begin{equation*}
		R_V (l, z, M; \tau) 
		 =
		 \frac{(2 \pi i)^{l+1}}{(l+1)!}
		\sum_{k=-1}^{l}
		\binom{l+1}{k+1} 
		(-z)^k B_{k+1}(\vec{y}; \tau) B_{l-k}(\vec{x}; \tau)
	\end{equation*}
Therefore 
	\begin{equation*}
		\lim_{\tau \rightarrow i \infty} R_V (l, z, M; \tau) 
 		=
		\frac{(2 \pi i)^{l+1}}{(l+1)!}
		\sum_{k=-1}^{l}
		\binom{l+1}{k+1}
		(-z)^k \B_{k+1}(y) \B_{l-k}(x)
	\end{equation*}
	where $\B_n(x)$ are bernoulli functions.
By tending $x$ and $y$ to $0$, we obtain the period polynomials over $\mathbb{Q}$.  
\begin{remark}
	When $\tau$ tends to $i \infty$,
		$R_V (m, z, M; \tau) $ are equal to $f_{l+1} (-d,c; z)$ in \cite[(1.2)]{carlitz1} up to a constant.
	Thus
		(\ref{4.lemma2a}) below can be considered as
		elliptic generalizations of Carlitz's reciprocity relations \cite[(1.11)]{carlitz1}.
\end{remark}

The transformation formulae are as follows:
\begin{theorem}\label{4.Theorem1}
	For $V = \bigl(
		\begin{smallmatrix}
        			a  &  b       \\
        			c  &  d    
		\end{smallmatrix} \bigr) \in \mathrm{SL}_2 (\mathbb{Z})$,
		let $\mathrm{M}_2 (V)$ be the subset of $\mathrm{M}_2 (\mathbb{R})$ 
		defined by
		\begin{equation}\label{4.definition_SM}
			\mathrm{M}_2 (V)
			:=
			\biggl{\{}
			\begin{pmatrix} 
				x' & x  \\ 
				y' & y
			\end{pmatrix}
			=
			\begin{pmatrix} 
				\vec{x}  \\ 
				\vec{y}
			\end{pmatrix} 
			\in \mathrm{M}_2 (\mathbb{R})
			\bigg{\vert}
			\vec{y}, c \vec{x}  + d \vec{y}\notin \mathbb{Z}^2
			\biggr{\}}.
		\end{equation}
	Let $\alpha$ be an irrational real algebraic number.
	If $l \in \mathbb{Z}$ with $l \geq 3$,
 		$V \in \mathrm{SL}_2 (\mathbb{Z})$
		and	$M \in \mathrm{M}_2(V)$, it holds that
	\begin{equation}\label{4.theorem1a}
		\HH (l, \alpha, M; \tau)
		-
 		j(V;\alpha)^{l-1}
 		\HH (l, V\alpha, VM; \tau)
		=
		R_V (l, \alpha, M; \tau).
	\end{equation}
\end{theorem}
To prove the theorem, 
	we need the following two lemmas 
	which will be proved in Section \ref{Sect_Lemma1} and \ref{Sect_Lemma2}.
\begin{lemma}\label{4.Lemma1}
	Let $T$ and $S$ denote the matrices
		$\begin{pmatrix} 1  &  1 \\ 0  &  1 \end{pmatrix}$
		and
		$\begin{pmatrix} 0  &  -1 \\1  &  0 \end{pmatrix}$
		respectively.  
	If $V \in \{  T^{\pm}, S \}$, then (\ref{4.theorem1a}) holds.
\end{lemma}
\begin{lemma}\label{4.Lemma2}
	If $l \in \mathbb{Z}_{>0}$, 
		$z \in \mathbb{C}$,
		$V, V_1, V_2 \in \mathrm{SL}_2 (\mathbb{Z})$ with $V = V_2 V_1$,
		and $M \in \mathrm{M}_2(V) \bigcap \mathrm{M}_2(V_1)$, then  
	\begin{equation}\label{4.lemma2a}
		R_{V_1} (l, z, M; \tau)
			+ j(V_1, z)^{l-1} R_{V_2} (l, V_1 z, V_1 M; \tau)
		=
		R_V (l, z, M; \tau).
	\end{equation}
\end{lemma}
We prove Theorem \ref{4.Theorem1}.
\begin{proof}[Proof of Theorem \ref{4.Theorem1}]
	Put 
		$G = \{  T^{\pm}, S \}$.
	For any positive integer $n$,
		let $U_n$ be a subset of $\mathrm{SL}_2 (\mathbb{Z})$ defined by
	\begin{equation*}
		U_n
		:=
		\{ V \in \mathrm{SL}_2 (\mathbb{Z}) \vert 
			\text{there are } n \text{ matrices } V_1, \ldots ,V_n \in G \text{ with } V = V_n V_{n-1} \cdots V_1
		\}.
	\end{equation*}
	If $V$ is a matrix in $\mathrm{SL}_2 (\mathbb{Z})$,
 		there is a positive integer $n$
		with  $V \in U_n$ since $G$ generates $\text{SL}_2 (\mathbb{Z}$) (see \cite[Theorem 2.1]{apostol1}).
	We will prove (\ref{4.theorem1a}) by induction on $n$.
	If $n$ equals $1$, the claim has been shown in Lemma \ref{4.Lemma1}.
	Suppose that it is true in case of $n-1$.
	Let $V$ be a matrix in $U_n$.
	Then there are two matrices $V_1$ and $V_2$ such that
		$V_1 \in G, V_2 \in U_{n-1}$ and $V = V_2 V_1$.
	By virtue of $j(V;\alpha) = j(V_1;\alpha)j(V_2;V_1\alpha)$,
	\begin{eqnarray*}
		& & 
		\HH (l, \alpha, M; \tau)
		-
 		j(V;\alpha)^{l-1}
 		\HH (l, V\alpha, VM; \tau) \\
 		& = & 
		\HH (l, \alpha, M; \tau)
		-
 		j(V_1;\alpha)^{l-1}
		\HH (l, V_1 \alpha, V_1 M; \tau) \\
		& & + 
		j(V_1;\alpha)^{l-1}
 		\Bigl(
			\HH (l, V_1 \alpha, V_1 M; \tau)
		-
 		j(V_2;V_1\alpha)^{l-1}
			\HH (l, V_2 V_1 \alpha, V_2  V_1 M; \tau))
		\Bigr) \\
		& \stackrel{\text{(induction)}}{=} & 
			R_{V_1} (l, \alpha, M; \tau) +  j(V_1, \alpha)^{l-1} R_{V_2} (l, V_1 \alpha, V_1 M; \tau)
		\\
		& \stackrel{(\text{\ref{4.lemma2a}})}{=}  & 
		 R_V (l, \alpha, M; \tau)
	\end{eqnarray*}
		which gives (\ref{4.theorem1a}) when $M \in \mathrm{M}_2(V) \bigcap \mathrm{M}_2(V_1)$.
	The set $\mathrm{M}_2(V) \bigcap \mathrm{M}_2(V_1)$ is dense in $\mathrm{M}_2 (\mathbb{R})$,
		so we obtain (\ref{4.theorem1a})
		with $M \in \mathrm{M}_2(V)$, which completes the proof.
\end{proof}

\section{A part of Arakawa's transformation formulae and Berndt's Theorem \ref{1.Berndt_Theorem} (ii)} \label{Sect_BerndtArakawaResults}
In this section, by use of Theorem \ref{4.Theorem1} we verify 
	a part of Arakawa's transformation formulae,
	or (\ref{1.Arakawa_Equation}) with integers $s < -1$.
Berndt's Theorem \ref{1.Berndt_Theorem} (ii) is induced from \cite[Proposition 2.2]{arakawa2}
	which is derived from Arakawa's transformation formulae with integers $s < -1$,
	thus we obtain Berndt's Theorem \ref{1.Berndt_Theorem} (ii) as a corollary.

Let $\mySin(x)$ and $\myCos(x)$ be $\sin(2 \pi x)$ and $\cos (2 \pi x)$ respectively.
We firstly introduce the behaviors of the real parts 
	of 
	$\dfrac{(l+1)!}{(2\pi i)^{l+1}} \HH (s,\alpha, M; \tau)$ 
	and 
	$S_{m,n}(r, x', x, y', y; \tau)$ when $\tau \rightarrow i \infty$.
To do this, we need the functions
	\begin{equation*}
		\begin{split}
			& Cl_l(x) 
			:= 
			\begin{cases} 
				\displaystyle{}
				\sum_{m = 1}^{\infty} \frac{\mySin(m x)}{m^l} & ( l \ even),  \\
				\displaystyle{}
				\sum_{m = 1}^{\infty} \frac{\myCos (m x)}{m^l} & ( l \ odd), 
			\end{cases}
			\\
			& S_{m,n}(r, x, y)
			:=
			\sum_{j (   d(r) )} 
			\B_m \Bigl(\frac{j + y}{d(r)} \Bigr)
			\B_n \Bigl(n(r) \frac{j + y}{d(r)} - x \Bigr), \\
			& \mathfrak{C}(r, x, y)
			:= 
			\frac{1}{c} \sideset{}{^\prime}\sum_{j(   c)}
				\cot \Bigl(\pi (\frac{j + y}{d(r)}) \Bigr) \cot \Bigl(\pi  (n(r) \frac{j + y}{d(r)} - x ) \Bigr),
		\end{split}
	\end{equation*}
	where the prime of the last summation means excluding $j$ 
	such that $a\frac{j+z}{c}-x \in \mathbb Z$ or $b\frac{j+z}{c}-y \in \mathbb Z$.
We note that 
	$Cl_l(x)$ are Clausen functions (see \cite{Lewin1}),
	and $S_{m,n}(r, x, y)$ and $\mathfrak{C}(r, x, y)$ respectively are 
	the generalized Dedekind-Rademacher sums in \cite{HWZ}
	and the cotangent sum in \cite{dieter1}:
	\begin{equation*}
	\begin{split}
		S_{m,n}(r, x, y) &= S_{m,n} \bigl( \begin{smallmatrix} 1 & n(r) & d(r) \\ 0 & x & y \end{smallmatrix} \bigr),\\
		\mathfrak{C}(r, x, y) &= \mathfrak{c}(1, n(r), d(r) ;0, x, y).
	\end{split}
	\end{equation*}
The behaviors of the real parts are following:
\begin{proposition}\label{DegenerationProperties}
	Let $ \mathrm{Re}\  z $ denote  the real part of a complex number $z$.\\
	{\bf (i)}
	Let $\alpha$ be an irrational real algebraic number,
		$l$ an integer with $l > 2$, 
	and $x', x, y', y$ real numbers with $(y', y) \notin \mathbb{Z}^2$.
	Then we have
	\begin{multline} \label{3.proposition1ii}
		\mathrm{Re} \Bigl( \frac{(l+1)!}{(2 \pi i)^{l+1}}  \lim_{\tau \rightarrow i \infty} \HH (l, \alpha, x', x, y', y; \tau) \Bigr) \\
		=
		\begin{cases}
			\displaystyle \frac{(l+1)!}{(2 \pi i)^{l}} \Bigl( \HH (l, \alpha, x, y) - \chi (y) Cl_l(x) \Bigr) &(l\ odd), \\
			\displaystyle \frac{(l+1)!}{(2 \pi i)^{l}} \Bigl( \HH (l, \alpha, x, y) - i \chi (y) Cl_l(x) \Bigr)  & (l\  even).
		\end{cases}
	\end{multline} 
	{\bf (ii)}
	Let $r$ be a rational number $r \neq 0$ and 
		$x', x, y', y$ real numbers with $\vec{x}, n(r) \vec{x} - d(r) \vec{y} \notin \mathbb{Z}^2$.
	Then we have
	\begin{multline} \label{2.proposition1ii}
		 \mathrm{Re} \lim_{\tau \rightarrow i \infty} S_{m,n}(r, x', x, y', y; \tau)\\
		=
		\begin{cases}
			S_{m,n}(r, x, y) - \frac{1}{4} \mathfrak{C}(r, x, y) & (m= n=1, y, n(r) y - d(r) x \in \mathbb{Z}), \\
			S_{m,n}(r, x, y) & (otherwise).
		\end{cases}
	\end{multline}
\end{proposition}	
\begin{proof}
	In order to prove (i),
		we introduce the Fourier expansion for the Jacobi form (see \cite[p. 70]{weil} or \cite[p. 456]{zagier1})
		\begin{equation*}
			F(\xi,X;\tau) := \frac{\theta '(0;\tau) \theta(\xi + X;\tau)}{\theta(\xi ;\tau) \theta(X;\tau)}.
		\end{equation*}
	If $\abs{\mathrm{Im}\ \xi}, \abs{\mathrm{Im}\ X} < \abs{\mathrm{Im}\ \tau}$, then
		\begin{equation*}
			F(\xi,X;\tau)
			=
			\pi \bigl( \cot \pi \xi + \cot \pi X \bigr)
				-2\pi i \sum_{i,j=1}^{\infty} (\e(i\xi+jX)- \e(-i\xi-jX)) \e(ij\tau).
		\end{equation*}
	Let $\xi', \xi, X'$ and $X$ be real numbers with $ \abs{X'} \leq 1/2$ and $ \abs{\xi} < 1$. 
	We find from the Fourier expansion that
	\begin{eqnarray}\label{3.proposition1aa}
		& &
		\F (\xi',\xi; X'\tau +X; \tau)\nonumber \\ 
		& = & 
		\e(\xi (X'\tau +X) ) F(-\xi' + \xi \tau, X'\tau +X ;\tau) \nonumber \\
		& = & 
		\pi  \e(\xi (X' \tau +X)) \Bigl( \cot \pi(-\xi'+\xi \tau) + \cot \pi(X+ X'\tau ) \Bigr)  \\
		& & 
			-2 \pi i \e(\xi X) \sum_{i,j=1}^{\infty} \Bigl( \e(-i\xi'+jX) \e((i + X')(j +\xi)\tau)  \nonumber \\
		& & \qquad \qquad \qquad \qquad \qquad
			- \e(i\xi'-jX) \e((i - X')(j - \xi)\tau) \Bigr).\nonumber
	\end{eqnarray}
	If $c = \min \{1 - \abs{\xi}, 1/2 \}$, 
		then it follows from $j \pm \xi \geq c j$ and $i \pm X' \geq ci$ $(i, j \in \mathbb{Z})$ that
		\begin{multline} \label{3.proposition1bb}
			\biggl{\lvert}\sum_{i,j=1}^{\infty} \Bigl( \e(-i\xi'+jX) \e((i + X')(j +\xi)\tau)\\  
				- \e(i\xi'-jX) \e((i - X')(j - \xi)\tau) \Bigr) \biggr{\rvert} 
			\leq 
			2  \sum_{i,j=1}^{\infty} \e (c^2 ij \tau).
		\end{multline}
	Let the symbols $\iinp{\xi}$ and $\langle \langle \xi \rangle \rangle$ be as in the proof of Lemma \ref{3.Lemma1},
		and $\{ x \}$ as in Section \ref{Sect_Introduction}.
	For any integer $m' \in \mathbb{Z}$,
	let $y_0 = \frp{-y}$ if $m = 0$,
		and let $y_{m'}$ be the real number such that
		$y_{m'} \equiv -y \pmod{1}$, 
		$\abs{y_{m'}} <1$ and $y_{m'} \ffrp{\alpha m'} \geq 0$ otherwise.
	It is derived from (\ref{2.key_formula1}), (\ref{2.key_formula2}) and
		$\alpha m' = \iinp{\alpha m'} + \ffrp{\alpha m'}$ that
		\begin{equation*}
			\F (-\vec{y}; \alpha (\tau m' +m); \tau)
			=  
			\e(-\iinp{\alpha m'} y') \F (-y',y_{m'};  \ffrp{\alpha m'} \tau + \alpha m; \tau),
		\end{equation*}
		thus we find from (\ref{3.proposition1aa}) and (\ref{3.proposition1bb}) that
		\begin{eqnarray} \label{3.proposition1cc}
			& & \lim_{\tau \rightarrow i \infty} \HH (l, \alpha, x', x, y', y; \tau) \nonumber \\
			& = & 
			\lim_{\tau \rightarrow i \infty} \pi
				\sideset{}{^\prime}\sum\limits_{m', m}
				\dfrac{\e (\vec{m} \cdot \vec{x})}{(\tau m' +m)^{l}} \e(-\iinp{\alpha m'} y') 
				\e(y_{m'} ( \ffrp{\alpha m'} \tau +\alpha m)) \\
			& & \qquad \qquad \times
				\Bigl(	 \cot \pi(y'+y_{m'} \tau) + \cot \pi(\alpha m+ \ffrp{\alpha m'} \tau )  \Bigr) \nonumber \\
			& \bigl( = &
			\lim_{\tau \rightarrow i \infty} \pi
				\sideset{}{^\prime}\sum\limits_{m', m} S_{m',m} \bigr) \nonumber
		\end{eqnarray}
		where $S_{m',m}$ are the summands in the summation.
	Let $c$ be a positive real number such that 
		$\abs{z \cot \pi z} \leq c$
		for any complex number $z$ with $\abs{\mathrm{Re}\ z}, \abs{\mathrm{Im}\  z} \leq 1/2$.
	Since $\cot \pi z = i (\e(z) + 1)/(\e(z)-1)$ converges at $\pm i$ when $z$ tends to $\mp i \infty$,
		there is a positive real number $c'$ such that
		\begin{equation*} \abs{ \cot \pi z} \leq \frac{c}{\abs{z}} + c' \end{equation*}
		for $z \in \mathbb{C} \setminus \{ 0 \}$  with $\abs{\mathrm{Re}\ z} \leq 1/2$.
	Set $C = \max \{ c, c'\}$.
	Since $y_{m'} \ffrp{\alpha m'}  \geq 0$, one obtains 
		\begin{equation*}
			\sideset{}{^\prime}\sum_{m', m \atop (m' \neq 0)} \abs{S_{m',m}}
			\leq
			C \sideset{}{^\prime}\sum_{m', m \atop (m' \neq 0)} \dfrac{1}{\abs{\tau m' +m}^{l}} 
			\Bigl( \frac{1}{\abs{\ffrp{y'}+ y_{m'} \tau}} + \frac{1}{\abs{\ffrp{\alpha m} +\ffrp{\alpha m'} \tau }} + 2 \Bigr) 
		\end{equation*}
		which together with (\ref{3.lemma1aa}) and $l > 2$ gives 
		\begin{equation}\label{3.proposition1dd}
			\lim\limits_{\tau \rightarrow i \infty} \sideset{}{^\prime}\sum_{m', m \atop (m' \neq 0)} S_{m',m} =0.
		\end{equation}
	Thus we have
		\begin{eqnarray*}
			& &  \lim_{\tau \rightarrow i \infty} \HH (l, \alpha, x', x, y', y; \tau) \\
			& = & 
			\pi \sideset{}{^\prime}\sum_m \lim_{\tau \rightarrow i \infty} 
				\frac{\e(mx)}{m^{l}} \e (m \frp{-y} \alpha)
				\Bigl( \cot \pi(y'+\frp{-y} \tau) + \cot \pi m \alpha \Bigr) \\
			& = & 
			\pi \times
			\begin{cases}
				\displaystyle
				\sideset{}{^\prime}\sum_m  \frac{\e(m x)}{m^{l}} \cot \pi m \alpha 
					+
					\cot \pi y' \sideset{}{^\prime}\sum_m \frac{\e(m x)}{m^{l}} ,
					& 
					\quad (y \in \mathbb{Z}), \\
				\displaystyle
				\sideset{}{^\prime}\sum_m  \frac{\e(m (x+ \{-y\} \alpha))}{m^{l}} \cot \pi m \alpha 
					-
					i \sideset{}{^\prime}\sum_m \frac{\e(x + \{ -y \} \alpha))}{m^{l}},
					& 
					\quad (y \notin \mathbb{Z}),
			\end{cases} \\
			& = & 
			2 \pi \times
			\begin{cases}
				\displaystyle
				\sum_{m = 1}^{\infty}  \frac{\myCos (m x)}{m^{l}} \cot \pi m \alpha 
					+
					i \cot \pi y' \sum_{m = 1}^{\infty} \frac{\mySin (m x)}{m^{l}},
					& 
					(y \in \mathbb{Z}, \ l\ odd), \\
				\displaystyle
				i \sum_{m = 1}^{\infty}  \frac{\mySin (m x)}{m^{l}} \cot \pi m \alpha 
					+
					\cot \pi y' \sum_{m = 1}^{\infty} \frac{\myCos (m x)}{m^{l}},
					& 
					(y \in \mathbb{Z}, \ l\ even), \\					
				\displaystyle
				\sum_{m = 1}^{\infty}  \frac{\myCos (m (x + \{ -y \} \alpha))}{m^{l}} \cot \pi m \alpha 
					+
					\sum_{m = 1}^{\infty} \frac{\mySin (m (x + \{ -y \} \alpha))}{m^{l}},
					& 
					(y \notin \mathbb{Z}, \ l\ odd), \\
				\displaystyle
				i \Bigl( \sum_{m = 1}^{\infty}  \frac{\mySin (m (x + \{ -y \} \alpha))}{m^{l}} \cot \pi m \alpha 
					-
					\sum_{m = 1}^{\infty} \frac{\myCos (m (x + \{ -y \} \alpha))}{m^{l}} \Bigr),
					& 
					(y \notin \mathbb{Z}, \ l\ even).
			\end{cases} 
		\end{eqnarray*}
	Set $\chi' (x) := 1 - \chi (x)$ where 
		$\chi (x)$ is the characteristic function of integers defined in Section \ref{Sect_Introduction},
		that is, $\chi' (x) = 0$ if $x \in \mathbb{Z}$ and $\chi' (x) = 1$ otherwise.
	By the above equation,
		we obtain an expression of the left hand side of (\ref{3.proposition1ii}) up to $(2 \pi i)^{l}/(l+1)!$:
		\begin{eqnarray*}
			& & 
			\frac{(2 \pi i)^{l}}{(l+1)!} \times
			\mathrm{Re} \Bigl( \frac{(l+1)!}{(2 \pi i)^{l+1}}  
				\lim_{\tau \rightarrow i \infty} \HH (l, \alpha, x', x, y', y; \tau) \Bigr) \\
			& = &
			\begin{cases}			
				\displaystyle
				- i \Bigl( \sum_{m = 1}^{\infty}  \frac{\myCos (m (x + \{ -y \} \alpha))}{m^{l}} \cot \pi m \alpha 
					+
					\chi' (y) \sum_{m = 1}^{\infty} \frac{\mySin (m (x + \{ -y \} \alpha))}{m^{l}} \Bigr),
					& 
					(l\ odd), \\
				\displaystyle
				\sum_{m = 1}^{\infty}  \frac{\mySin (m (x + \{ -y \} \alpha))}{m^{l}} \cot \pi m \alpha 
					-
					\chi' (y)  \sum_{m = 1}^{\infty} \frac{\myCos (m (x + \{ -y \} \alpha))}{m^{l}},
					& 
					(l\ even).
			\end{cases} 
		\end{eqnarray*}
	On the other hand, if $y \notin \mathbb{Z}$, 
		then it follows from (\ref{2.infiniteExpression_H}) and $\langle -y \rangle = \{ -y \}$ that
		\begin{eqnarray*}
			& & 
			\HH (l, \alpha, x, y) \\
			& = &
			\frac{1}{2} \sum_{m \in \mathbb{Z} \atop (m \neq 0)}
				\frac{\e(m (x + \langle -y \rangle \alpha))}{m^s} 
				\Bigl( \frac{\e(\alpha m) + 1}{\e(\alpha m) -1} - 1 \Bigr)\\
			& = &
			\frac{1}{2i} \sum_{m \in \mathbb{Z} \atop (m \neq 0)}
				\frac{\e(m (x + \langle -y \rangle \alpha))}{m^s} 
				\cot \pi m \alpha 
				-
				\frac{1}{2}
					\sum_{m \in \mathbb{Z} \atop (m \neq 0)}
					\frac{\e(m (x + \langle -y \rangle \alpha))}{m^s} 
				\\			
			& = & 
			\begin{cases}			
				\displaystyle
				- i \Bigl( \sum_{m = 1}^{\infty}  \frac{\myCos (m (x + \{ -y \} \alpha))}{m^{l}} \cot \pi m \alpha 
					+
					\sum_{m = 1}^{\infty} \frac{\mySin (m (x + \{ -y \} \alpha))}{m^{l}} \Bigr),
					& 
					(l\ odd), \\
				\displaystyle
				\sum_{m = 1}^{\infty}  \frac{\mySin (m (x + \{ -y \} \alpha))}{m^{l}} \cot \pi m \alpha 
					-
					\sum_{m = 1}^{\infty} \frac{\myCos (m (x + \{ -y \} \alpha))}{m^{l}},
					& 
					(l\ even).
			\end{cases} 
		\end{eqnarray*}
	If $y \in \mathbb{Z}$, then it does from (\ref{1.definition_H}), (\ref{1.definition_generalized_H}) and $\langle -y \rangle = 1$ that 
		\begin{eqnarray*}
			& & 
			\HH (l, \alpha, x, y) \\
			& = &
			\sum_{m=1}^{\infty} \frac{\e(mx)}{m^{l}} \frac{\e (m \alpha)}{1-\e(m \alpha)}
				+
				\e(\frac{1-l}{2})
				\sum_{m=1}^{\infty} \frac{\e(-mx)}{m^{l}} \frac{\e (m \alpha)}{1-\e(m \alpha)}\\
			& = &
			\frac{1}{2} \sum_{m=1}^{\infty} \frac{\e(mx)}{m^{l}} \Bigl( \frac{\e(\alpha m) + 1}{\e(\alpha m) -1} + 1 \Bigr)
				+
				\frac{1}{2} \sum_{m=1}^{\infty} \frac{\e(-mx)}{(-m)^{l}} \Bigl( \frac{\e(- \alpha m) + 1}{\e(- \alpha m) -1} - 1 \Bigr)
				\\
			& = & 
			\begin{cases}			
				\displaystyle
				- i \sum_{m = 1}^{\infty}  \frac{\myCos (m x)}{m^{l}} \cot \pi m \alpha 
					+
					Cl_l(x),
					& 
					\quad (l\ odd), \\
				\displaystyle
				\sum_{m = 1}^{\infty}  \frac{\mySin (m x)}{m^{l}} \cot \pi m \alpha 
					+
					i Cl_l(x),
					& 
					\quad (l\ even).
			\end{cases} 
		\end{eqnarray*}	
		Thus we get (i).

	We will verify (ii) next.
	By \cite[PROPOSITION 11]{machide} and (\ref{2.relation}),
		we obtain
		\begin{eqnarray*}
			& & \mathrm{Re} \lim_{\tau \rightarrow i \infty} S_{m,n}(r, x', x, y', y; \tau) \\
			& = &
			\begin{cases}
				S_{m,n}(r, x, y; \tau) - \frac{1}{4} \mathfrak{C}(r, x, y) & (m= n=1, (0, x, y) \in (1, n(r), d(r)) \mathbb{R} + \mathbb{Z}^3), \\
				S_{m,n} \bigl(\begin{smallmatrix} x  \\ y\end{smallmatrix};r \bigr) & (otherwise).
			\end{cases}
		\end{eqnarray*}
	By virtue of \cite[LEMMA 8]{machide},
		the condition $(0, x, y) \in (1, n(r), d(r)) \mathbb{R} + \mathbb{Z}^3$ is equivalent to 
		the condition
		$y, n(r) y - d(r) x \in \mathbb{Z}$.
	Thus we obtain (ii) .
\end{proof}

We prove the following proposition, or the remainder of the tasks in this section.
\begin{proposition}
	Theorem \ref{4.Theorem1} induce (\ref{1.Arakawa_Equation}) with integers $s$ less than $-1$ 
		if $\tau \rightarrow i \infty$.
\end{proposition}
\begin{proof}
	For any integer $l$, let $\psi (l)$ denote $1$ if $l$ is odd, and $i$ if $l$ is even.
	By tending $\tau$ to $i \infty$ in (\ref{4.theorem1a}) up to $(l+1)!/(2 \pi i)^{l+1}$ 
		and taking its real part,
		we obtain by Proposition \ref{DegenerationProperties} that
		\begin{eqnarray*}
			& & 
			\frac{(l+1)!}{(2 \pi i)^{l}}
			\Biggl[
				\Bigl(
					\HH (l, \alpha, x, y) - \psi(l) \chi(y) Cl_l(x)
				\Bigl) \\
			& & 
			\quad 
				-
	 			j(V;\alpha)^{l-1}
				\Bigl(
	 				\HH (l, V\alpha, ax + by, cx + dy) - \psi(l) \chi(cx + dy) Cl_l(ax + by)
				\Bigr)
			\Biggr]\\
			& = &
			\sum_{k=-1}^{m} \binom{m+1}{k+1} 
				(-j(V;\alpha))^k S_{k+1,m-k} ( \frac{d}{c},-x, y).
		\end{eqnarray*}
	On the other hand, we can derive the above equation 
		from (\ref{1.Arakawa_Equation}) 
		with integers $s < -1$ and
			$(x, y) = (q', -p') = (q, -p) \bigl( \begin{smallmatrix} d & -c \\ -b & a\end{smallmatrix} \bigr)$
		in a similar way as the proof of \cite[Proposition 2.2]{arakawa2}.
	We omit the rest of the proof because of similarity.

\end{proof}

\section{Proof of Lemma \ref{4.Lemma1}} \label{Sect_Lemma1}
We give a proof of Lemma \ref{4.Lemma1} in this section.
Let $\sideset{}{}\sum\limits_{m', m}$ (resp. $\sideset{}{^\prime}\sum\limits_{m', m}$).
	means running over all integers $(m', m) \in \mathbb{Z}^2$ (resp. $\mathbb{Z}^2 \setminus \{ (0,0) \})$.
If there is the symbol "$e$" on them,
	we promise that the summation rule follows Cauchy principal value, i.e.,
	\begin{equation}\label{sumRules}
	\begin{split}
		\sideset{}{_e}\sum_{m', m} &= \lim_{M \rightarrow \infty} \sum_{m'=-M}^M \sum_{m = -M}^M, \\
		\sideset{}{_e^\prime}\sum_{m', m} &= \lim_{M \rightarrow \infty} \underset{(m', m) \neq (0, 0)}{\sum_{m'=-M}^M \sum_{m = -M}^M}.
	\end{split}
	\end{equation}
In order to prove Lemma \ref{4.Lemma1},
	we need the equation in the following proposition:
\begin{proposition}\label{5.Proposition1}
	Let $\alpha$ be an irrational real algebraic number,
		and $s$ be a complex number with $\mathrm{Re}(s) >3$.
	If $x', x, y', y \in \mathbb{R} \setminus \mathbb{Z}$, then
		\begin{equation}\label{5.proposition1a}
			\Bigl( 
			\sideset{}{_e^\prime}\sum_{m', m} \sideset{}{_e^\prime}\sum_{n', n} -\sideset{}{_e^\prime}\sum_{n', n} \sideset{}{_e^\prime}\sum_{m', m}
			\Bigr)
			\frac{\e(m'y' + my)}{(\tau m' + m)^s} 
				\frac{\e(n'x' + nx)}{\alpha (\tau m' + m) +\tau n' +n}
			=
			0,
		\end{equation}
		or the order of $(m', m)$ and $(n',n)$ in the sum can be changed.
\end{proposition}
We give the proof of Lemma \ref{4.Lemma1} 
before verifying (\ref{5.proposition1a}).
\begin{proof}[Proof of Lemma \ref{4.Lemma1}]
	Let 
		$\bigl(
		\begin{smallmatrix} 
			x' & x  \\
			y' & y 
		\end{smallmatrix} \bigr)
		=
		\bigl(
		\begin{smallmatrix} 
			\vec{x}  \\ 
			\vec{y}
		\end{smallmatrix} \bigr)
 		\in \mathrm{M}_2 (V)$.
	It follows from (\ref{2.key_formula1}) that
		\begin{equation*}
			\HH \bigl(l, T \alpha,
			T
			\bigl(
			\begin{smallmatrix} 
				\vec{x}  \\ 
				\vec{y}
			\end{smallmatrix}
			\bigr)
			;\tau \bigr)
			=
			\HH \bigl(l, \alpha+1,
			\bigl(
			\begin{smallmatrix} 
				\vec{x} + \vec{y} \\ 
				\vec{y}
			\end{smallmatrix}
			\bigr)
			;\tau \bigr)
			=
			\HH \bigl(l, \alpha, 
			\bigl(
			\begin{smallmatrix} 
				\vec{x}  \\ 
				\vec{y}
			\end{smallmatrix}
			\bigr)
			;\tau \bigr),
		\end{equation*}
	thus we obtain (\ref{4.theorem1a}) in case of $V = T^{+}$.
Replacing $\alpha$ by $T^{-} \alpha$ on the above equation shows the case $V = T^{-}$.

We will prove (\ref{4.theorem1a}) in case of $V =S$. 
Let $l$ be an integer with $l \geq4$,
	and $x', x, y', y \in \mathbb{R} \setminus \mathbb{Z}$. 
Since 
	$Y^l -X^l = (Y-X) \sum_{k=0}^{l-1} X^k Y^{l-1-k}$,
	\begin{equation*}
		\frac{1}{X^l} \frac{1}{X-Y} 
			+  \frac{1}{Y^l} \frac{1}{Y-X} 
		=
		- \sum_{k=0}^{l-1} \frac{1}{Y^{k+1}} \frac{1}{X^{l-k}}.
	\end{equation*}	
	By replacing $Y$ by $- \frac{1}{\alpha} Y$
	and multiplying $\alpha^{-1}$ both sides, 
	we get
	\begin{equation*}
		\frac{1}{X^l} \frac{1}{\alpha X+Y} 
		- \alpha^{l-1} \frac{1}{(-Y)^l} \frac{1}{\frac{1}{\alpha}Y+X} 
		=
		\sum_{k=0}^{l-1} (-\alpha)^k \frac{1}{Y^{k+1}} \frac{1}{X^{l-k}}.
	\end{equation*}
	By substituting $\tau m' +m$ and $\tau n'+ n$ for $X$ and $Y$ respectively,
		and  ranging $(m', m)$ and $(n', n)$ over $\mathbb{Z}^2 \setminus \{(0,0)\}$ 
		up to $\e(m'x'+mx)\e(n'y'+ny)$ as the  rules (\ref{sumRules}),
		we obtain
		\begin{multline}\label{5.proof_of_lemma4_1aa}
			\sideset{}{_e^\prime}\sum_{n', n} \sideset{}{_e^\prime}\sum_{m', m}
			\Bigl(
			\frac{\e(m'x'+mx)}{(\tau m' + m)^l} 
				\frac{\e(n'y'+ny)}{\alpha (\tau m' + m) +\tau n' +n}
			\\
			-
			\alpha ^{l-1}
			\frac{\e(n'y'+ny)}{(-\tau n' - n)^l}
				\frac{\e(m'x'+mx)}{\frac{1}{\alpha} (\tau n' + n) +\tau m' +m}
			\Bigr)\\
			=
			\sum_{k=0}^{l-1} (-\alpha)^k 
			\sideset{}{_e^\prime}\sum_{n', n} \frac{\e(n'y'+ny)}{(\tau n' + n)^{k+1}} 
 			\sideset{}{_e^\prime}\sum_{m', m} \frac{\e(m'x'+mx)}{(\tau m' + m)^{l-k}} .
		\end{multline}
	On the other hand, L. Kronecker \cite{kronecker1} (see \cite{weil} for a proof) showed that
		\begin{equation*}\label{5.proof_of_lemma4_1bb}
		\begin{split}
			\F (\vec{x};X;\tau)
			&=
			\sideset{}{_e}\sum_{m',m}
			\frac{\e(-m'x'-mx)}{X +\tau m' +m}
			=
			- \sideset{}{_e}\sum_{m',m}
			\frac{\e(m'x'+mx)}{-X +\tau m' +m},
			\\
			B_k(\vec{x};\tau)
			&=
			-\frac{k!}{(2\pi i)^k}
			\sideset{}{_e^\prime}\sum_{m',m}
			\frac{\e(m'x'+mx)}{(\tau m' +m)^k}.
		\end{split}
		\end{equation*}
	Thus, by (\ref{5.proposition1a}),
		the left hand side of (\ref{5.proof_of_lemma4_1aa}) equals 
		\begin{eqnarray*}
			& &
			\sideset{}{_e^\prime}\sum_{m', m}
				\frac{\e(m'x'+mx)}{(\tau m' + m)^l} 
				\Bigl(		
					\sideset{}{_e}\sum_{n', n} 			
					\frac{\e(n'y'+ny)}{\alpha (\tau m' + m) +\tau n' +n}
					-
					\frac{1}{\alpha (\tau m' + m)}			
				\Bigr)
			\\
			& &
			-
			\alpha ^{l-1}
			\sideset{}{_e^\prime}\sum_{n', n} 
				\frac{\e(n'y'+ny)}{(-\tau n' - n)^l}
				\sideset{}{_e}\sum_{m', m}
				\Bigl(
					\frac{\e(m'x'+mx)}{\frac{1}{\alpha} (\tau n' + n) +\tau m' +m}
					-
					\frac{1}{\frac{1}{\alpha} (\tau n' + n)}
				\Bigr)
			\\
			& = &
 			\sideset{}{_e^\prime}\sum_{m', m}
			\frac{\e(m'x'+mx)}{(\tau m' + m)^l} 
				\F (- \vec{y}; \alpha (\tau m' + m); \tau) 
			+ \frac{1}{\alpha}
			\frac{(2\pi i)^{l+1}}{(l+1)!} B_{l+1}(\vec{x};\tau) 
			\\
			& &
			-
			\alpha ^{l-1}  \Bigl(
			\sideset{}{_e^\prime}\sum_{n', n}
			\frac{\e(-n'y'-ny)}{(\tau n' + n)^l} 
				\F(-\vec{x}; -\frac{1}{\alpha} (\tau n' + n);\tau)
			+ (-1)^l \alpha
			\frac{(2\pi i)^{l+1}}{(l+1)!} B_{l+1}(\vec{y};\tau)
			\Bigr).
		\end{eqnarray*}
	We note that one can omit the symbol "$e$" on $\sideset{}{_e^\prime}\sum_{}$ above
		because the series are absolute converges if $l \geq 4$.
	On the other hand, the right hand side of (\ref{5.proof_of_lemma4_1aa}) equals 
		\begin{eqnarray*}
			& &
			\frac{(2\pi i)^{l+1}}{(l+1)!}
				\sum_{k=0}^{l-1} (-\alpha)^k \binom{l+1}{k+1} B_{k+1} (\vec{y};\tau) B_{l-k} (\vec{x};\tau)
		\\
		&=&
			R_S (l, \alpha, M; \tau)
			-
			\frac{(2 \pi i)^{l+1}}{(l+1)!}
			\Bigl(
				-
				\frac{1}{\alpha} B_{l+1} (\vec{x};\tau)
				+
				(-\alpha)^{l} B_{l+1} (\vec{y};\tau)
			\Bigr)
		\end{eqnarray*}
	These yield (\ref{4.theorem1a}) with $l \geq 4$ and $V =S$
		since the set 
		$\bigl{\{} 
		\bigl(
		\begin{smallmatrix}
			x' & x  \\ 
			y' & y
		\end{smallmatrix} \bigr) 
		\big{\vert}
		x', x, y', y \in \mathbb{R} \setminus \mathbb{Z}
		\bigr{\}}$
		is dense in $\mathrm{M}_2 (V)$.
	In order to get (\ref{4.theorem1a}) with $l=3$ and $V = S$,
		we need differential equations for Kronecker's double series
		and their generating function as follows.
		\begin{eqnarray*}
			& &	
			( \tau \frac{\partial}{\partial x '} + \frac{\partial}{\partial x }) 
			\F (\vec{x}; X;\tau)
			 =
			2 \pi i X \F (\vec{x}; X;\tau), \\
			& &	
			( \tau \dfrac{\partial}{\partial x '} + \dfrac{\partial}{\partial x }) 
				 B_m (\vec{x};\tau)
			=
			m B_{m-1} (\vec{x};\tau)
		\end{eqnarray*}
		which are derived from the definitions (\ref{2.definition_K}) and (\ref{3.definition_B}) immediately.
	If $ \mathrm{Re}\  s >3$ and $(z', z) \in \{ (x', x), (y', y) \}$, 
		we can easily see from the proof of Lemma \ref{3.Lemma1} that the series
		\begin{equation*}
			\sideset{}{^\prime}\sum\limits_{m', m}
			(\tau \frac{\partial}{\partial z'} + \frac{\partial}{\partial z})
			\Bigl(
			\frac{\e (m'y'+my)}{(\tau m' +m)^{s}}
			\F (\vec{x}; \alpha (\tau m' +m); \tau)
			\Bigr)
		\end{equation*}
		absolutely converges,
		thus the term wise differentiation of the above series is possible. 
	Therefore applying the differential operator $\tau \dfrac{\partial}{\partial x '} + \dfrac{\partial}{\partial x }$
		to (\ref{4.theorem1a}) with $l = 4$
		yields (\ref{4.theorem1a}) with $l=3$.
\end{proof}
To give a proof of Proposition \ref{5.Proposition1},
	we need the following lemma.
Its proof is based on Siegel's way \cite[pp. 31--32]{siegel1}.
We note that it was used in \cite{katayama1} for proving Rmanujan's formulas with respect to $L$-functions. 
\begin{lemma}\label{5.Lemma1}
	Let $ y', y, \lambda$ be real numbers, 
		and $N_1', N_1, N_2', N_2, m', m$ integers.
	If $y', y \notin \mathbb{Z}, \lambda >0, 
		N_1' \leq N_2', N_1 \leq N_2 $ and $ (m', m) \neq	 (0, 0)$,
		then there is a positive real number $D = D(y', y, \lambda, \alpha, \tau)$
			dependent to $y', y, \lambda, \alpha, \tau$ 
			and 
			independent to $N_1', N_1, N_2', N_2, m', m$
			such that
			\begin{equation} \label{5.lemma1a}
				\Abs{
					\sum_{n' = N_1'}^{N_2'} \sum_{n = N_1}^{N_2}
					\frac{\e(n'y' + ny)}{\alpha (\tau m' +m) + \tau n' +n}
				} 
				\leq
				D \abs{\tau m' + m}^{\lambda +1}.
			\end{equation} 
\end{lemma}
\begin{proof}
	We consider the four cases as follows:
		\begin{equation*}
			\begin{split}
				 \mathrm{(i)}  & -\alpha m' \notin [N_1', N_2'] ,\quad -\alpha m \notin [N_1, N_2], \\
				 \mathrm{(ii)}  & -\alpha m' \in [N_1', N_2'], \quad -\alpha m \notin [N_1, N_2], \\
				 \mathrm{(iii)}  & -\alpha m' \notin [N_1', N_2'], \quad -\alpha m \in [N_1, N_2], \\
				 \mathrm{(iv)}  & -\alpha m' \in [N_1', N_2'], \quad -\alpha m \in [N_1, N_2].
			\end{split}
		\end{equation*}
	Firstly, we verify (\ref{5.lemma1a}) in case of (i). 
	Set the real numbers $z, a_{n'}' , a_{n}, f_{n', n}$ as 
		\begin{equation*}
		\begin{split}
				& a_{n'}' = \frac{\e (n'y')}{\e(y')-1},\quad a_{n} = \frac{\e (ny)}{\e(y)-1},\\
				& z = \alpha (\tau m' +m),\quad f_{n', n} = z + \tau n' +n.
		\end{split}
		\end{equation*}
	Note that $\e(y')-1 \neq 0$ and $\e(y)-1 \neq 0$ because of $y', y \notin \mathbb{Z}$.
	Since $\e (n'y') = a_{n'+1}' - a_{n'}'$ and
		$\e (ny) = a_{n+1} - a_{n}$,
		we find that
		\begin{eqnarray*}
		& &
			\sum_{n' = N_1'}^{N_2'} \sum_{n = N_1}^{N_2} \frac{\e(n'y' + ny)}{z + \tau n' +n}
		\\
		& = &
			\sum_{n' = N_1'}^{N_2'} \sum_{n = N_1}^{N_2} \frac{a'_{n'+1}a_{n+1}-a'_{n'+1}a_{n}-a'_{n'}a_{n+1}+a'_{n'}a_{n}}{f_{n',n}}
		\\
		& = &
			\sum_{n' = N_1' +1}^{N_2' +1} \sum_{n = N_1 +1}^{N_2+1} \frac{a'_{n'}a_{n}}{f_{n'-1,n-1}}
			-
			\sum_{n' = N_1'+1}^{N_2'+1} \sum_{n = N_1}^{N_2} \frac{a'_{n'}a_{n}}{f_{n'-1,n}}
			\\
			& &
			\hspace{0pt}
			-
			\sum_{n' = N_1'}^{N_2'} \sum_{n = N_1+1}^{N_2+1} \frac{a'_{n'}a_{n}}{f_{n',n-1}}
			+
			\sum_{n' = N_1'}^{N_2'} \sum_{n = N_1}^{N_2} \frac{a'_{n'}a_{n}}{f_{n',n}}
		\\
		& = &
			\sum_{n' = N_1'+1}^{N_2'} \sum_{n = N_1+1}^{N_2} a'_{n'}a_{n}
				\Bigl( \frac{1}{f_{n'-1,n-1}}-\frac{1}{f_{n'-1,n}}-\frac{1}{f_{n',n-1}}+\frac{1}{f_{n',n}} \Bigr)
			\\
			& &
			\hspace{0pt}
			-
			a_{N_1} \sum_{n' = N_1'+1}^{N_2'} a'_{n'} \Bigl(\frac{1}{f_{n'-1,N_1}}-\frac{1}{f_{n',N_1}} \Bigr)
			+
			a_{N_2+1} \sum_{n' = N_1'+1}^{N_2'} a'_{n'} \Bigl(\frac{1}{f_{n'-1,N_2}}-\frac{1}{f_{n',N_2}} \Bigr)
			\\
			& &
			\hspace{0pt}
			-
			a_{N_1'} \sum_{n= N_1+1}^{N_2} a_{n} \Bigl(\frac{1}{f_{N_1',n-1}}-\frac{1}{f_{N_1',n}} \Bigr)
			+
			a_{N_2'+1} \sum_{n= N_1+1}^{N_2} a_{n} \Bigl(\frac{1}{f_{N_2',n-1}}-\frac{1}{f_{N_2',n}} \Bigr)
			\\
			& &
			\hspace{0pt}
			+
			\frac{a'_{N_2'+1}a_{N_2+1}}{f_{N_2',N_2}}
			-
			\frac{a'_{N_2'+1}a_{N_1}}{f_{N_2',N_1}}
			-
			\frac{a'_{N_1'}a_{N_2+1}}{f_{N_1',N_2}}
			+
			\frac{a'_{N_1'}a_{N_1}}{f_{N_1',N_1}}.
		\end{eqnarray*}
	If $\xi', \xi, \xi_1', \xi_1, \xi_2', \xi_2$ are real numbers such that 
		\begin{equation*}
			\xi_1' \leq \xi_2', \quad
				\xi_1 \leq \xi_2, \quad
 				-\alpha m' \notin [\xi_1', \xi_2'], \quad
				-\alpha m \notin [\xi_1, \xi_2],
		\end{equation*}
 		then elementary integral theory tells us that
 		\begin{equation*}
 			\begin{split}
 				& \frac{1}{f_{\xi_1', \xi}} -  \frac{1}{f_{\xi_2', \xi}} = \tau \int_{\xi_1'}^{\xi_2'} \frac{d \mu'}{(z+ \tau \mu' +\xi)^2}, \\
				& \frac{1}{f_{\xi', \xi_1}} - \frac{1}{f_{\xi', \xi_2}} = \int_{\xi_1}^{\xi_2} \frac{d \mu}{(z+ \tau \xi' +\mu)^2}, \\
 				&  \frac{1}{f_{\xi_1', \xi_1}} -  \frac{1}{f_{\xi_1', \xi_2}} -  \frac{1}{f_{\xi_2', \xi_1}} + \frac{1}{f_{\xi_2', \xi_2}}
				 	= 2\tau \int_{\xi_1'}^{\xi_2'} \int_{\xi_1}^{\xi_2} \frac{d \mu'  d \mu }{(z+ \tau \mu' +\mu)^3}.
 			\end{split}
 		\end{equation*}
	Let $A = A(y', y, \tau)$ be $2(1+\abs{\tau})/\abs{\e(y')-1} \abs{\e(y)-1}$
		and $LHS$ be the left hand side of (\ref{5.lemma1a}).
	Since $2(1+\abs{\tau})\abs{a_{n'}'}\abs{a_n}  \leq A$
		if $N_1' \leq n' \leq N_2' + 1$ and $N_1 \leq n \leq N_2 + 1$,
		it follows that
		\begin{eqnarray}\label{5.lemma1cc}
			LHS
		&\leq&
			A \Bigl(
				\int_{N_1'}^{N_2'} \int_{N_1}^{N_2}
 				\frac{d \mu'  d \mu }{\abs{z+ \tau \mu' +\mu}^3} 
			\\
			& &
			\hspace{20pt}
			+
			\int_{N_1'}^{N_2'} \frac{d \mu'}{\abs{z+ \tau \mu' +N_1}^2} 
			+
			\int_{N_1'}^{N_2'} \frac{d \mu'}{\abs{z+ \tau \mu' +N_2}^2} 
			\nonumber
			\\
			& &
			\hspace{20pt}
			+
			\int_{N_1}^{N_2} \frac{d \mu}{\abs{z+ \tau N_1' +\mu}^2}
			+
			\int_{N_1}^{N_2} \frac{d \mu}{\abs{z+ \tau N_2' +\mu}^2}
			\nonumber
			\\
			& &
			\hspace{20pt}
			+
			\frac{1}{\abs{f_{N_1', N_1}}}
			+
			\frac{1}{\abs{f_{N_1', N_2}}}
			+
			\frac{1}{\abs{f_{N_2', N_1}}}
			+
			\frac{1}{\abs{f_{N_2', N_2}}}
			\Bigr).
			\nonumber
		\end{eqnarray}
	Put $t' = \alpha m'$ and $t = \alpha m$. 
	It follows from (\ref{3.lemma1aa}) that
		\begin{eqnarray*}
		& &
			\int_{N_1'}^{N_2'} \int_{N_1}^{N_2}
 				\frac{d \mu'  d \mu }{\abs{z+ \tau \mu' +\mu}^3} 
		\\	
		&\leq&
			\frac{1}{\bigl(2 (\lvert \tau \rvert - \lvert \mathrm{Re}\ \tau \rvert) \bigr)^{3/2}}
				\int_{N_1'}^{N_2'} \int_{N_1}^{N_2}
 				\frac{d \mu'  d \mu }{\abs{t' + \mu'}^{3/2} \abs{t + \mu}^{3/2}}.
		\end{eqnarray*}
	Thus, because $-t' \notin [N_1', N_2']$ and $-t \notin [N_1, N_2]$ by the condition (i), 
		there are a real number  $b = b(\tau)$ depending on $\tau$ such that
		\begin{eqnarray*}
		& &
			\int_{N_1'}^{N_2'} \int_{N_1}^{N_2}
 			\frac{d \mu'  d \mu }{\abs{z+ \tau \mu' +\mu}^3} 
		\\	
		&\leq&
			b \Bigl(
			 \frac{1}{\abs{t' + N_2'}^{1/2} \abs{t + N_2}^{1/2}}
			 	+  \frac{1}{\abs{t' + N_2'}^{1/2} \abs{t + N_1}^{1/2}} 
			\\
			& & \hspace{20pt}
				+  \frac{1}{\abs{t' + N_1'}^{1/2} \abs{t + N_2}^{1/2}}
				+  \frac{1}{\abs{t' + N_1'}^{1/2} \abs{t + N_1}^{1/2}}
			\Bigr ).
		\end{eqnarray*}
	Since
		\begin{equation*}
			\int \frac{1}{y^2 + x^2} dx
			=
			\frac{1}{y} \arctan \frac{x}{y},
		\end{equation*}
		we have
		\begin{eqnarray*}
		& &
			\int_{N_1}^{N_2} \frac{d \mu}{\abs{z+ \tau N' +\mu}^2}
		%
		\\
		&=&
			\int_{N_1}^{N_2} \frac{d \mu}{\bigl(t+\mu+(\mathrm{Re}\ \tau)(t'+N') \bigr)^2 +(\mathrm{Im}\ \tau)^2(t'+N')^2}
		\\
		&=&
			\frac{1}{ \abs{\mathrm{Im}\ \tau} \abs{t'+N'}}
				\biggl[
					\arctan 
						\frac{t+\mu+ (\mathrm{Re}\ \tau) (t'+N')} {\abs{\mathrm{Im}\ \tau} \abs{t'+N'}}
				\biggr]_{\mu= N_1}^{\mu = N_2}
		\\
		&\leq&
			\frac{\pi}{ \abs{\mathrm{Im}\ \tau} \abs{t'+N'}}.
		\end{eqnarray*}
	In a similar way, we get		
		\begin{equation*}
			\int_{N_1'}^{N_2'} \frac{d \mu'}{\abs{z+ \tau \mu' +N}^2}
		%
		\leq
			\frac{\pi}{\abs{\tau}^2 \abs{\mathrm{Im}\ \frac{1}{\tau}} \abs{t+N}}.
		\end{equation*}
	Furthermore we easily see from (\ref{3.lemma1aa}) that
		\begin{equation*}
			\frac{1}{\abs{f_{N_{j'}', N_j}}}
			=
			\frac{1}{\abs{\tau t'+t+\tau N_{j'}' + N_{j}}}
			\leq
			\Bigl( \frac{1}{2(\lvert \tau \rvert - \lvert \mathrm{Re}\ \tau \rvert)\abs{t'+N_{j'}'}\abs{t+N_{j}}} \Bigr)^{1/2}
		\end{equation*}
		if $j', j = 1,2$.
	Therefore we conclude from (\ref{5.lemma1cc}) that 
		there is a positive real number $B=B(y', y,\tau)$ depending on $y', y,\tau$ such that 
		\begin{multline*}
			LHS
			\leq
			B \Bigl(
				\bigl(\frac{1}{\abs{\alpha m' + N_1'}^{1/2}}+ \frac{1}{\abs{\alpha m' + N_2'}^{1/2}} \bigr)
				\bigl(\frac{1}{\abs{\alpha m + N_1}^{1/2}}+ \frac{1}{\abs{\alpha m + N_2}^{1/2}} \bigr) \\
				+
				\frac{1}{\abs{\alpha m' + N_1'}} + \frac{1}{\abs{\alpha m' + N_2'}}
				+
				\frac{1}{\abs{\alpha m + N_1}} + \frac{1}{\abs{\alpha m + N_2}}
			\Bigr).
		\end{multline*}
	On the other hand, by the theorem of Thue-Siegel-Roth in
		the diophantine approximation theory,
		there is a positive real number $C = C (\alpha, \lambda) > 1$ depending on $\alpha, \lambda$ such that 
		\begin{equation*}
			\Abs{\alpha - \frac{k}{l}} > \frac{1}{C l^{\lambda +2}} \quad ( k, l \in \mathbb{Z}, \ l  > 0).
		\end{equation*}
	We see from this that
		\begin{equation*}
			\frac{1}{\abs{\alpha l + k}} 
			<
			C \abs{l}^{\lambda +1}
		\end{equation*}
		for $k, l \in \mathbb{Z}$ with $(k,l) \neq (0,0)$,
		which deduces  
		from $C \abs{l}^{\lambda +1}< C ( \abs{l}^{\lambda +1} +1) \leq C ( \abs{l}^{(\lambda +1)/2} +1)^2$
		that
		\begin{equation*}
			\frac{1}{\abs{\alpha l + k}^{1/2}} 
			<
			C^{1/2} ( \abs{l}^{(\lambda +1)/2} +1).
		\end{equation*}		
	Thus, by (\ref{3.lemma1aa}), we have
		\begin{eqnarray*}
			& &
			LHS
			\\
			& \leq &
			B
			\Bigl(
				4C(\abs{m'}^{(\lambda +1)/2} + 1)(\abs{m}^{(\lambda +1)/2} + 1)
				+
				2C(\abs{m'}^{\lambda +1} + \abs{m}^{\lambda +1} + 2)
			\Bigr)
			\\
			&\leq&
			8BC \Bigl(
				\abs{m'}^{(\lambda +1)/2} \abs{m}^{(\lambda +1)/2}
				+
				\abs{m'}^{\lambda +1} + \abs{m}^{\lambda +1} +1
				\Bigr)
			\\
			& \leq &
			8BC \Bigl(
				\frac{1}{\{2(\lvert \tau \rvert - \lvert \mathrm{Re}\ \tau \rvert)\}^{(\lambda+1)/2}}
				+
				\frac{2}{\mathrm{Im}(\tau)^{\lambda + 1}}
				+
				C'
				\Bigr) 
				\abs{\tau m' +m}^{\lambda + 1}
		\end{eqnarray*}
		where $C'$ is the maximum of $\{ 1/(\tau k' + k)  \vert (k', k) \in \mathbb{Z}^2 \setminus (0,0) \}$.
 	Because $B, C$ and $C'$ are dependent to only $y', y, \lambda, \alpha, \tau$,
		we obtain (\ref{5.lemma1a}) in case of (i).
		
	We consider the case (ii).
	We can easily obtain (\ref{5.lemma1a}) if $(m', N_1', N_2) = (0, 0, 0)$,
		so we assume that  $(m', N_1', N_2) \neq (0, 0, 0)$.
	Let $[x]$ denote the integer part of a real number $x$, i.e., $0 \leq x- [x] < 1$.
	Suppose that  $m' = 0$. 
	Then $ -\alpha m' \notin [N_1', -1], - \alpha m' \notin [1, N_2']$, $- \alpha m' \in [0,0]$.
	Set $g_{n',n}:= \e(n'y' + ny)/ (\alpha (\tau m' +m) + \tau n' +n)$.
	We have
		\begin{eqnarray*}
			\Abs{
				\sum_{n' = N_1'}^{N_2'} \sum_{n = N_1}^{N_2} g_{n',n}
			}
		\leq
			 \Abs{
			 	\sum_{n' = N_1'}^{-1} \sum_{n = N_1}^{N_2} g_{n',n}
			 }
			 +
			 \Abs{
			 	\sum_{n' = 1}^{N_2'} \sum_{n = N_1}^{N_2} g_{n',n}
			 }
			 +
			 \Abs{
			 	\sum_{n = N_1}^{N_2} g_{0,n}
			 }
		\end{eqnarray*}
		where empty sums mean $0$. 
	We see from the proof of (i) that the first and second terms in the last formula
		have the upper bound $\lvert \tau m'+m \rvert^{\lambda+1}$ up to constants depending only $y',y,\lambda, \alpha, \tau$,
		and from Siegel's way \cite[pp. 31--32]{siegel1} (which is also used in \cite[proof of Lemma 2]{katayama1}) that
		the third term 
		$\lvert \sum_{n = N_1}^{N_2} g_{0,n} \rvert = \lvert \sum_{n = N_1}^{N_2} \e(ny)/(m+n) \rvert $
		have the upper bound $\lvert m \rvert^{\lambda+1}$ up to constant depending only $y, \lambda$.
	Therefore we verify the case (ii) with $m'=0$.
	If $m' \neq 0$, then 
		$-\alpha m' \notin [N_1', [-\alpha m']]$ and $-\alpha m' \notin [ [-\alpha m'] +1, N_2']$,
		thus
		we can prove this case as same as the above case that $m'=0$.
	Therefore  we completes the proof of (\ref{5.lemma1a}) in case of (ii).

	We can prove similarly the other cases (iii) and (iv) as the case (ii), thus we omit the proofs.
\end{proof}

In the end of the section, We prove Proposition \ref{5.Proposition1}.
\begin{proof}[Proof of Proposition \ref{5.Proposition1}]
	Since $\mathrm{Re}\ s > 3$, there are positive real numbers $\lambda$ and $s_0$ 
		with $\mathrm{Re}\ s - \lambda -1 > s_0>2$.
	Let $\epsilon$ be an arbitrary positive real number.
	Since the series $\displaystyle{} \sideset{}{^\prime}\sum\limits_{m', m} \dfrac{1}{\abs{\tau m' +m}^{s_0}}$
		converges, there is a positive integer $M$ satisfying 
		\begin{equation*}
			\sum_{m' \geq M \atop or\ m \geq M}   \dfrac{1}{\abs{\tau m' +m}^{s_0}}
			 < 
			 \epsilon
		\end{equation*}
	where the summation runs over $(m', m) \in \mathbb{Z}$ such that $m' \geq M$ or $m \geq M$.
	Moreover since the series $\displaystyle{} \sideset{}{_e^\prime}\sum\limits_{n', n} \dfrac{\e (n'y' + ny)}{\alpha (\tau m' + m) +\tau n' +n}$ converges
		for every $(m', m) \in \mathbb{Z}^2 \setminus (0, 0)$,
		there is a positive integer $N_0$ such that
		\begin{equation*}
			 \quad
			\Abs{\sum_{\abs{n'}  \geq N \atop or\  \abs{n} \geq N} 
				\frac{\e(n'y' + ny)}{\alpha (\tau m' + m) +\tau n' +n}} < \epsilon
		\end{equation*}
		if $N > N_0$ and $(m', m) \in \{ \mathbb{Z}^2 \setminus (0,0) \  \vert \ \abs{m'}, \abs{m} \leq M \}$.
	Thus we find from (\ref{5.lemma1a}) that
		there is a positive real number $D$ not depending on $N$ and $\epsilon$ such that
 		\begin{eqnarray*}\label{5.proposition1aa}
			& &
			\Abs{
				\Bigl( 
				\sideset{}{_e^\prime}\sum_{m', m} \sideset{}{_e^\prime}\sum_{n', n} -\sideset{}{^\prime}
					\sum_{\abs{n'}< N \atop and\  \abs{ n} <N} \sideset{}{_e^\prime}\sum_{m', m}
				\Bigr)
				\frac{\e(m' x' + mx)}{(\tau m' + m)^s} 
					\frac{\e(n'y' + ny)}{\alpha (\tau m' + m) +\tau n' +n}
			}\\
			&= &
			\Abs{
				\Bigl(
					\sum_{\abs{m'}  \geq M \atop or\  \abs{m} \geq M} 
					+
					\sideset{}{^\prime}\sum_{\abs{m'}< M \atop and\  \abs{ m} <M}
				\Bigr)
				\sum_{\abs{n'}  \geq N \atop or\  \abs{n} \geq N} 
					\frac{\e(m' x' + mx)}{(\tau m' + m)^s} 
					\frac{\e(n'y' + ny)}{\alpha (\tau m' + m) +\tau n' +n}
			}\\
			& \leq &
			D \sum_{\abs{m'}  \geq M \atop or\  \abs{m} \geq M} 
				\frac{1}{\abs{\tau m' + m}^s} 
					\abs{\tau m' + m}^{\lambda +1} 
				+
				\epsilon
					\sideset{}{^\prime}\sum_{\abs{m'}< M \atop and\  \abs{ m} <M}
					\frac{1}{\abs{\tau m' + m}^s} \\
			& \leq &
			\epsilon
				\Bigl(
					D + \sideset{}{^\prime}\sum_{m' , m } \frac{1}{\abs{\tau m' + m}^s}
				\Bigr)
		\end{eqnarray*}
		if $N > N_0$.
		By virtue of the arbitrariness of $\epsilon$, we obtain (\ref{5.proposition1a}).
\end{proof}


\section{Proof of Lemma \ref{4.Lemma2}} \label{Sect_Lemma2}
In order to prove Lemma \ref{4.Lemma2},
	we give the following formulae 
	for the elliptic Dedekind-Rademacher sums $S_{1,l} (r,M; \tau)$.	

\begin{theorem}\label{6.Theorem1}
	Let $r $ be a rational number, and $l$ a positive integer.
	The maps $n: \mathbb{Q} \to \mathbb{Z}$ and $d: \mathbb{Q} \to \mathbb{Z}_{>0}$ are
		as (\ref{3.Def_ndNumber}). 
	Let $\mathrm{M}_2 (r)$ be the subset of $\mathrm{M}_2 (\mathbb{R})$ 
		defined by
		\begin{equation}\label{6.definition_SM}
			\mathrm{M}_2 (r)
			:=
			\biggl{\{}
			\begin{pmatrix} 
				x' & x  \\ 
				y' & y
			\end{pmatrix}
			=
			\begin{pmatrix} 
				\vec{x}  \\ 
				\vec{y}
			\end{pmatrix} 
			\in \mathrm{M}_2 (\mathbb{R})
			\bigg{\vert}
			\vec{x}, d(r) \vec{x} - n(r) \vec{y} \notin \mathbb{Z}^2
		\biggr{\}}.
		\end{equation}
	The relation of $\mathrm{M}_2 (V)$ defined in (\ref{4.definition_SM})
		is that $\mathrm{M}_2 (-d/c) = \mathrm{M}_2 (V)$ if $c \neq 0$.
	
	If $V = \bigl(
		\begin{smallmatrix}
		        a  &  b       \\
		        c  &  d    
		\end{smallmatrix} \bigr) \in \mathrm{SL}_2 (\mathbb{Z})$ with
		 $j(V; r) \neq 0$ and 
		 $M  \in \mathrm{M}_2( V) \bigcap \mathrm{M}_2 (r)$, 
		 then 
		\begin{multline}\label{6.theorem1a}
			S_{1,l} (r, M;  \tau)
				- j(V; r)^{l-1}
				S_{1,l} (V r, VM; \tau)\\
				= 
				\frac{(-1)^{l} l!}{(2\pi i)^{l+1}}
				R_V(l, r, M; \tau)
					-
					\frac{(-1)^{l} l}{l+1}
					\frac{c}{j (V;r)}
					\frac{B_{l+1}(d(r) \vec{x} - n(r) \vec{y};\tau)}{d(r)^{l+1}}.
		\end{multline}
\end{theorem}
\begin{remark}
Theorem \ref{6.Theorem1} reproduces a part of Halbritter's result \cite[Theorem 2]{halbritter} when $\tau \rightarrow i \infty$.	
\end{remark}
For any 
	$V = \bigl( 
	\begin{smallmatrix}
	          a & b      \\
	          c & d   
	\end{smallmatrix} \bigr) 
	\in \mathrm{SL}_2 (\mathbb{Z})$,
	set $\tilde{j} (V; r) := \dfrac{c}{j(V;r)} = \dfrac{d}{dr} j(V;r)$.	
Before proving the theorem, 
 	we give the proof Lemma \ref{4.Lemma2} by use of (\ref{6.theorem1a}).
\begin{proof}[Proof of Lemma \ref{4.Lemma2}]
	Because $\det V = ad - bc = 1$ and $Vr = (ar + b)/(cr + d) = (a n(r) + b d(r))/(c n(r) + d d(r) )$, 
		we have
		\begin{equation*}
			\gcd (a n(r) + b d(r), c n(r) + d d(r) ) =1.
		\end{equation*}
	Since $j(V;r)=cr+d=(cn(r)+dd(r))/d(r)$,
		we obtain
		\begin{equation}\label{6.proof_Lemma4.2aa}
			\begin{pmatrix} n(Vr)  \\ d(Vr) \end{pmatrix}
				=
				\mathrm{sgn}(j(V; r)) V \begin{pmatrix} n(r)  \\ d(r) \end{pmatrix}
				\end{equation}
		where $\mathrm{sgn}(j(V; r))$ equals $1$ if $j(V; r) >0$, and $-1$ if $j(V; r) < 0$.
	We also get 	
		\begin{equation}\label{6.proof_Lemma4.2bb}
			n(r) \eta - d(r) \xi
			=
			\mathrm{sgn}(j(V; r)) \bigl( n(Vr) (c\xi +d\eta) - d(Vr) (a\xi +b\eta) \bigr)
		\end{equation}
		for complex numbers $\xi, \eta$ by virtue of
		\begin{equation*} 
			\det
				\begin{pmatrix} n(V r)  &  a\xi +b\eta \\ d(V r)  &  c\xi +d\eta    \end{pmatrix} 
			=
			\mathrm{sgn}(j(V; r))  \det
				 V \begin{pmatrix} n(r)  &  \xi \\ d(r)  &  \eta     \end{pmatrix}.
		\end{equation*}		
	Since
	\begin{eqnarray*}
		0  
		&=& 
		\Bigl{\{} 
				S_{1,l} (r, M;  \tau) -  j(V; r)^{l-1} S_{1,l} (V r, VM; \tau) 
			\Bigr{\}}
		\\
		& &
			- 
			 \Bigl{\{} 
				S_{1,l} (r, M;  \tau) -  j(V_1; r)^{l-1} S_{1,l} (V_1 r, V_1M; \tau)  
		\\
		& &\qquad
				+
				j(V_1; r)^{l-1} 
					\Bigl(
						S_{1,l} (V_1 r, V_1M; \tau)  - j(V_2; V_1 r)^{l-1} S_{1,l} (V_2 V_1 r, V_2 V_1 M; \tau) 
					\Bigr)
			\Bigr{\}},
	\end{eqnarray*}
	we find from (\ref{6.theorem1a}), (\ref{6.proof_Lemma4.2aa}) and (\ref{6.proof_Lemma4.2bb}) that 
		\begin{eqnarray*}\label{6.proof_Lemma4.2cc}
			& & 
			 \frac{(l+1)!}{(2\pi i)^{l+1} l}
			 \Bigl( R_V (l, r,  M; \tau)- R_{V_1} (l, r, M; \tau) 
				- j(V_1, r)^{l-1} R_{V_2} (l, V_1 r, V_1 M ; \tau) \Bigr) \\
			& = &
			\tilde{j} (V; r)  \frac{ B_{l+1}(d(r) \vec{x} - n(r) \vec{y} ;\tau)}{d(r)^{l+1}} 
			-
			\tilde{j} (V_1; r) \frac{ B_{l+1}(d(r) \vec{x} - n(r) \vec{y} ;\tau)}{d(r)^{l+1}} \\
			&  &
				\qquad - j(V_1;r)^{l-1} \tilde{j} (V_2; V_1r) \frac{ B_{l+1}(d(V_1 r) (a_1 \vec{x} + b_1 \vec{y}) - n(V_1 r) (c_1 \vec{x} + d_1 \vec{y}) ;\tau)}{d(V_1r)^{l+1}} \\
			& = &
			 \frac{ B_{l+1}(d(r) \vec{x} - n(r) \vec{y} ;\tau)}{d(r)^{l+1}} 
				\Bigl{(}
					\tilde{j} (V; r) - \tilde{j} (V_1; r) 
						- \frac{1}{j(V_1, r)^{2} } \tilde{j} (V_2; V_1r) 
				\Bigr{)} 
		\end{eqnarray*}
		where $V_1 = \bigl( \begin{smallmatrix} a_1 & b_1 \\ c_1 & d_1\end{smallmatrix} \bigr)$.
	By differentiating $\log j(V;r) = \log j(V_1;r) j(V_2;V_1r)$ on $r$, 
		we obtain 
		\begin{equation*}
			\tilde{j} (V; r) - \tilde{j} (V_1; r)  - \frac{1}{j(V_1, r)^{2} } \tilde{j} (V_2; V_1r) = 0.
		\end{equation*}
	Thus (\ref{4.lemma2a}) follows if $z=r$ is any rational number.
	Since $R_V (l, z, M; \tau)$ are polynomials in $z$,
		(\ref{4.lemma2a}) holds for any complex number $z \in \mathbb{C}$.
\end{proof}
In order to verify (\ref{6.theorem1a}),
	we establish transformation formulae of the functions $\widehat{S}_{1,l}(r, M;   X  ,Y ; \tau)$ defined below
	instead of $S_{1,l} (r,M; \tau)$. 
The coefficients of $X^0, Y^0$ on the functions $\widehat{S}_{1,l}(r, M; X, Y ; \tau)$ 
	nearly equal $S_{1,l} (r,M; \tau)$ (see Lemma \ref{6.Lemma3} below),
	thus the transformation formulae of $\widehat{S}_{1,l}(r, M;  X, Y ; \tau)$ give rise to those of  $S_{1,l} (r,M; \tau)$.
		
Let us define $\widehat{S}_{1,l}(r, M;  X , Y ; \tau)$.
Assume that
	\begin{equation*}
		l \in \mathbb{Z}_{>0}, \quad
		r = n(r)/ d(r) \in \mathbb{Q}, \quad		
		M = \bigl(\begin{smallmatrix} x' & x \\ y' & y \end{smallmatrix} \bigr) 
			= \bigl(\begin{smallmatrix} \vec{x}  \\ \vec{y} \end{smallmatrix} \bigr) 
				\in \mathrm{M}_2 (\mathbb{R})
	\end{equation*}
	and $X, Y$ are complex variables.
If $M \in \mathrm{M}_2 (r)$,
	then $\widehat{S}_{1,l}(r, M; X, Y ; \tau)$ are defined by
	\begin{multline}\label{8.def_S}
		\widehat{S}_{1,l}(r, M;  X  , Y ; r) 
		\\
		:= 
		\frac{1}{d(r)}
			\sum_{j', j (  d(r))}
			\F \Bigl(\frac{\vec{j} + \vec{y}}{d(r)}; n(r) Y -d(r) X;  \tau \Bigr) 
				\F^{(l-1)} \Bigl(n(r) \frac{\vec{j} + \vec{y}}{d(r)} -\vec{x};-Y; \tau \Bigr)
	\end{multline}
	where
	$\F ^{(m)} (\vec{x};X;\tau)$ denotes
	$\dfrac{1}{(2\pi i)^m} \Bigl ( \dfrac{\partial}{\partial X} \Bigr )^m
		\F (\vec{x};X;\tau)$.
	We see from (\ref{2.key_formula1}) that $\widehat{S}_{1,l}(r; M,  X , Y ; r)$ 
		have the following periodicities with respect to $X$ and $Y$:
		\begin{equation}\label{6.periodicity_S}
			\begin{split}
			\widehat{S}_{1,l}(r, M;   X+1  , Y ; \tau)
			& = 
			\e(-y) \widehat{S}_{1,l}(r, M;   X  , Y ; \tau), \\
			\widehat{S}_{1,l}(r, M;  X  + \tau , Y ; \tau)
			& =
			\e(-y') \widehat{S}_{1,l}(r, M;   X  , Y ; \tau), \\			
			\widehat{S}_{1,l}(r, M;  X,  Y+1 ; \tau)
			& =
			\e(x) \widehat{S}_{1,l}(r, M;  X  , Y ; \tau),  \\
			\widehat{S}_{1,l}(r, M;  X  , Y+\tau ; \tau)
			& =
			\e(x') \widehat{S}_{1,l}(r, M;   X  , Y ; \tau). 
			\end{split}
		\end{equation}

For the purpose of establishing 
	 transformation formulae of
	$\widehat{S}_{1,l}(r, M;  X , Y ; \tau)$,
	we also need functions 
	$\widehat{R}_V (l, r, M;  X  , Y  ; \tau)$
	whose coefficients of $X^0,Y^0$ nearly equal $R_V (l, r, M; \tau)$ (see Lemma \ref{6.Lemma3}):
If $V = \bigl( 
	\begin{smallmatrix}
	        a  &  b       \\
	        c  &  d    
	\end{smallmatrix} \bigr) \in \mathrm{SL}_2 (\mathbb{Z})$
	and $M 
		= \bigl( \begin{smallmatrix} \vec{x}  \\ \vec{y}\end{smallmatrix} \bigr)  \in \mathrm{M}_2( V)$, 
	then the functions $\widehat{R}_V (l, r, M; X  , Y ; \tau)$ are defined by
	\begin{multline} \label{6.definition_R}
		\widehat{R}_V (l, r, M; X  , Y ; \tau)\\
		:=
		\begin{cases} \displaystyle{}
			(-1)^l \sum_{k=0}^{l-1}
			\binom{l-1}{k} 
			\bigl(
				- j (V, r)\bigr)^k
				\widehat{S}_{k+1,l-k} \biggl(\frac{d}{c}, \bigl(\begin{smallmatrix} -\vec{x}  \\ \vec{y} \end{smallmatrix} \bigr); X  , Y ; \tau \biggr) 
			& ( c \neq 0), \\
			0 & (c =0),
		\end{cases}
	\end{multline}
	where $\widehat{S}_{k+1,l-k} \biggl(\dfrac{d}{c},\bigl(\begin{smallmatrix} -\vec{x}  \\ \vec{y} \end{smallmatrix} \bigr); X,  Y ; \tau \biggr) $ 
	denotes
	\begin{equation}
		\frac{1}{c} \sum_{j', j (  \abs{c})} 
			\F^{(k)} (\frac{\vec{j} +\vec{y}}{c}; -cX -dY ;\tau)
			\F^{(l-1-k)} (d\frac{\vec{j} +\vec{y}}{c} + \vec{x}; Y ;\tau)
	\end{equation}
	and the summation takes over all elements in $(\mathbb{Z} / \abs{c} \mathbb{Z})^2$.
We see from (\ref{2.key_formula2}) that
	\begin{equation}
		\widehat{R}_{-V} (l, r, M; X  , Y ; \tau)
		=
		\widehat{R}_V (l, r, M; X  , Y ; \tau).
	\end{equation}
For a proof of the transformation formulae of $\widehat{S}_{1,l}(r, M; X , Y ; \tau)$,
we prepare two lemmas.
\begin{lemma}\label{6.Lemma1}
	Assume that $r = n(r)/ d(r) \in \mathbb{Q}$,
		$\bigl(\begin{smallmatrix} a  &  b  \\ c  &  d    \end{smallmatrix} \bigr) \in \mathrm{SL}_2(\mathbb{Z})$
		and
		$\bigl(\begin{smallmatrix} m'  &  n' \\m  &  n     \end{smallmatrix} \bigr)$,
		$\bigl(\begin{smallmatrix} m_0'  &  n_0' \\ m_0  &  n_0  \end{smallmatrix} \bigr) \in \mathrm{M}_2(\mathbb{Z})$.
	If $\bigl( \begin{smallmatrix} m'  &  n' \\ m  &  n    \end{smallmatrix} \bigr) 
		=  \bigl( \begin{smallmatrix} m_0'  &  n_0' \\ m_0 &  n_0  \end{smallmatrix} \bigr)
			\bigl( \begin{smallmatrix} a  &  b \\ c  &  d    \end{smallmatrix} \bigr)$,
		then 
		\begin{equation}\label{6.lemma1a}
			r (\tau m' + m) + \tau n' + n 
			=
			(ar +b) (\tau m_0' +m_0) + (cr +d) (\tau n_0' +n_0).
		\end{equation}
	In particular, we obtain the equalities of lattice
		\begin{equation}\label{6.lemma1b}
			\begin{split}
			\frac{1}{d(r)} (\tau \mathbb{Z} +\mathbb{Z}) 
			& =
			r (\tau \mathbb{Z} +\mathbb{Z}) + \tau \mathbb{Z} +\mathbb{Z} \\
			& =
			(ar +b)  (\tau \mathbb{Z} +\mathbb{Z}) 
				+
				(c r +d)  (\tau \mathbb{Z} +\mathbb{Z}).
			\end{split}
		\end{equation}
\end{lemma}
\begin{proof}
	A direct calculation verifies (\ref{6.lemma1a}).
	The first equality in (\ref{6.lemma1b}) follows from 
		the fact that $\gcd (n(r), d(r)) =1$.
	The second from (\ref{6.lemma1a}) and 
		$\mathrm{det} \bigl(\begin{smallmatrix} a  &  b \\ c  &  d  \end{smallmatrix} \bigr) = 1$.
\end{proof}
\begin{lemma}\label{6.Lemma2}
	Let $\vec{g} \cdot \vec{h} := g'h' + gh$ be the inner product
		for two vectors $\vec{g} = (g', g)$ and $\vec{h} = (h', h) \in \mathbb{R}^2$.
	Assume that $c \in \mathbb{Z} \setminus \{0\}$, $l \in \mathbb{Z}_{\geq 0}$, 
		and $\vec{i} = (i', i) \in \mathbb{Z}^2$.
	Then we have
		\begin{eqnarray}
			& &
			\F^{(l)} (\vec{x}; X + \frac{\tau i' +i}{c};\tau)
			=
			c^{l-1} \sum_{j', j ( \abs{c})} 
				\e(\vec{i}\cdot  \frac{\vec{j}+\vec{x}}{c})
				\F^{(l)}(\frac{\vec{j}+\vec{x}}{c}; cX; \tau) , \label{6.lemma2a} 
			\qquad
			\\
			& &
			B_m (\vec{x}; \tau)
			=
			c^{m-2} \sum_{j', j ( \abs{c})} 
				B_m (\frac{\vec{j}+\vec{x}}{c}; \tau), \label{6.lemma2b}
		\end{eqnarray}
		where the summations take over all elements in $(\mathbb{Z} / \abs{c} \mathbb{Z})^2$.
\end{lemma}
\begin{proof} 
	Suppose that $\vec{j} = (j',j) \in \mathbb{Z}^2$ and 
 		$X \notin \dfrac{\tau}{c} \mathbb Z + \dfrac{1}{c}\mathbb Z$.
	Set
		\begin{equation*}
			f(z) = \F (\vec{x}; -z +X; \tau) \F(\frac{\vec{j}+\vec{x}}{c};cz;\tau).
		\end{equation*}
	We find from (\ref{2.key_formula1}) that
		the function $f(z)$ is a doubly periodic function on $\tau \mathbb{Z} +\mathbb{Z}$,
		and from the positions of the poles of $\F (\vec{x}; X; \tau)$
		that
		$f(z)$ has the only simple poles on the lattices
 		$\dfrac{\tau}{c} \mathbb Z + \dfrac{1}{c}\mathbb Z$
 		and $X + \tau \mathbb Z + \mathbb Z$.
	Since the sum of the residues of $f(z)$ at its poles in 
 		any period parallelogram equals zero,
 		we have
		\begin{equation*}
			\frac{1}{c} \sum_{k', k (   \abs{c})} 
			\F(\vec{x}; X +\frac{\tau k' +k}{c};\tau) \e(-k'\frac{j'+x'}{c}-k\frac{j+x}{c})
			=
			\F (\frac{\vec{j}+\vec{x}}{c}; cX;\tau).
		\end{equation*}
	By adding each side of the above equation for $j', j =0, 1, \ldots, \abs{c}-1$,
		we obtain
		\begin{equation*}
			c \F (\vec{x}; X ;\tau)
			=
 			\sum_{j', j (   \abs{c})} \F(\frac{\vec{j}+\vec{x}}{c}; cX; \tau)
		\end{equation*}
		because $\sum\limits_{j = 0}^{\abs{c}-1} \e ( -k \dfrac{j}{\abs{c}} )$ equals $\abs{c}$ if $k$ divides $\abs{c}$, 
			or equals $0$ otherwise. 
	Replacing $X$ by $X + \dfrac{\tau i' +i}{c}$ and (\ref{2.key_formula1})
		imply (\ref{6.lemma2a}).
	The second equation (\ref{6.lemma2b}) follows from (\ref{6.lemma2a}) with $i' = i =0$ 
		and (\ref{3.definition_B}).
\end{proof}

For describing the transformation formulae of
	$\widehat{S}_{1,l}(r, M;  X , Y ; \tau)$,
	we use the following notations 
	which are only different in the description of variables $X$ and $Y$:
	\begin{equation}
		\begin{split}
			&
			\widehat{S}_{1,l}(r, M;  \begin{smallmatrix} X \\ Y \end{smallmatrix} ; \tau)
			:=
			\widehat{S}_{1,l}(r, M;  X , Y ; \tau), \\
			&
			\widehat{R}_V (l, r, M; \begin{smallmatrix} X  \\ Y \end{smallmatrix} ; \tau)
			:=
			\widehat{R}_V (l, r, M; X  , Y ; \tau)
		\end{split}
	\end{equation}
	
The transformation formulae are the following.
\begin{proposition}\label{6.Proposition1}
	Let
		$ l \in \mathbb{Z}_{>0}$, 
		$r = n(r)/ d(r) \in \mathbb{Q}$,
		$V \in \mathrm{SL}_2 (\mathbb{Z})$,
		$M \in\mathrm{M}_2( V) \bigcap \mathrm{M}_2 (r)$,
		and $X, Y \in \mathbb{C}$.
	If $j(V; r) > 0$,
		then we have
		\begin{equation}\label{6.proposition1a}
			\widehat{S}_{1,l}(r, M;  \begin{smallmatrix} X \\ Y \end{smallmatrix} ; \tau)
				- j(V; r)^{l-1} \widehat{S}_{1,l}(V r, V M;  V \bigl( \begin{smallmatrix} X \\ Y \end{smallmatrix} \bigr) ; \tau)
			=
			\widehat{R}_V (l, r, M; \begin{smallmatrix} X  \\ Y \end{smallmatrix} ; \tau).
		\end{equation}	
	Here the maps $n, r$ and the sets $\mathrm{M}_2( V)$, $\mathrm{M}_2 (r)$ 
		are as (\ref{3.Def_ndNumber}), (\ref{4.definition_SM}) and (\ref{6.definition_SM}).  
\end{proposition}
\begin{proof}
	Set $V = \bigl( \begin{smallmatrix} a  &  b \\ c  &  d \end{smallmatrix} \bigr)$
		and
		$M = \bigl( \begin{smallmatrix} \vec{x}  \\ \vec{y} \end{smallmatrix} \bigr) 
		=
		\bigl( \begin{smallmatrix} x' & x  \\ y' & y \end{smallmatrix} \bigr)$.
	Let $M$ and $Y$ be fixed.
	We define the function $L(X), R(X)$ and $LR(X)$ by
		\begin{equation*}
			\begin{split}
				L(X)  & := \text{(the left hand side of (\ref{6.proposition1a}))}, \\
				R(X) & := \text{(the right hand side of (\ref{6.proposition1a}))}, \\
				LR(X) & := L(X)- R(X). 
			\end{split}
		\end{equation*}
	We find from (\ref{2.key_formula1}) and (\ref{6.periodicity_S}) that 
		$L(X)$ and $R(X)$, or $LR(X)$, have the quasi periodicity depending on $y', y$ on $\tau \mathbb{Z} +\mathbb{Z}$:
		\begin{equation}\label{6.proposition1aa}
			LR(X+1) =\e(-y) LR(X),\quad
			LR(X+\tau) = \e(-y') LR(X).
		\end{equation}
		
	The goal of the proof is $LR(X) = 0$.
	For the goal, it is enough to prove that $LR(X)$ is an entire function when
		\begin{equation*}
			Y \notin \dfrac{1}{d(V r)}  (\tau \mathbb{Z} +\mathbb{Z}), \quad
			y', y \notin \mathbb{Z}.
		\end{equation*}
	In fact, if it is true, 
		then $LR(X)$ is a bounded function by the periodicity,
		thus $LR(X)$ is a constant function 
		for above $Y, y', y$ because of Liouville's theorem.
	Since $y', y \notin \mathbb{Z}$ and (\ref{6.proposition1aa}), the constant should be equal to $0$. 
	By the continuity of $LR(X)$ as a function of $Y, y', y$,
		we may omit the restrictions of $Y, y', y$.
		
	We will attain the goal.
	We define 
		$M_0
			= \bigl( \begin{smallmatrix} \vec{x_0}  \\ \vec{y_0} \end{smallmatrix} \bigr) 
			= \bigl( \begin{smallmatrix} x_0' & x_0  \\ y_0' & y_0 \end{smallmatrix} \bigr) $
		and
		$\bigl( \begin{smallmatrix} X_0  \\ Y_0 \end{smallmatrix} \bigr) $
		by
	\begin{equation}\label{7.eqn_MXY}
	\begin{split}
		M_0 
		=
		\begin{pmatrix} \vec{x_0}  \\  \vec{y_0} \end{pmatrix}
		&: = 
		V M
		=
		\begin{pmatrix}  a\vec{x}+b\vec{y} \\   c\vec{x} +d\vec{y} \end{pmatrix}
		= 
		\begin{pmatrix} a x' + by' &  ax+by \\ cx' + dy'  &  cx +dy \end{pmatrix},
		\\
		\begin{pmatrix} X_0  \\ Y_0 \end{pmatrix}
		&:= 
		V \begin{pmatrix} X  \\ Y \end{pmatrix} 
		= 
		\begin{pmatrix} a X + b Y \\ c X + d Y \end{pmatrix}.
	\end{split}
	\end{equation}
	Since $j(V;r) > 0$,
		we see from (\ref{6.proof_Lemma4.2aa}) and (\ref{6.proof_Lemma4.2bb}) that 
		\begin{equation}\label{6.proposition1cc}
			\begin{split}
			& 
			\widehat{S}_{1,l}(V r, V M;  V \bigl( \begin{smallmatrix} X \\ Y \end{smallmatrix} \bigr) ; \tau)
			\\
			=&
			\widehat{S}_{1,l}(\frac{n(Vr)}{d(Vr)},
					\begin{pmatrix} \vec{x_0} \\ \vec{y_0} \end{pmatrix};  \begin{pmatrix} X_0 \\ Y_0 \end{pmatrix} ; \tau) 
			\\
			=& 
			\frac{1}{d(V r)} \sum_{j', j (  d(Vr))}
				\F \Bigl(\frac{\vec{j} + \vec{y_0}}{d(Vr)}; n(r) Y -d(r) X;  \tau \Bigr)
			\\
			& \qquad \qquad \qquad \qquad \times
				\F^{(l-1)} \Bigl(n(V r) \frac{\vec{j} + \vec{y_0}}{d(V r)} -\vec{x_0};-cX-dY; \tau \Bigr).
			\end{split}
		\end{equation}
	Therefore it follows from (\ref{8.def_S}), (\ref{6.definition_R}) and (\ref{6.proposition1cc}) 
		that all possible poles of the function $L(X)$ are on the lattices
		\begin{equation*}
			\begin{cases}
	 			r Y + \dfrac{1}{d(r)} (\tau \mathbb{Z} + \mathbb{Z}), \quad
	 				- \dfrac{d}{c} Y + \dfrac{1}{c} (\tau \mathbb{Z} + \mathbb{Z}) & ( c \neq 0),\\
	  			r Y + \dfrac{1}{d(r)} (\tau \mathbb{Z} + \mathbb{Z})  & (c =0),
			\end{cases}
		\end{equation*}
		and that those of the function $R(X)$ are on
		\begin{equation*}
  			- \dfrac{d}{c} Y + \dfrac{1}{c} (\tau \mathbb{Z} + \mathbb{Z})
		\end{equation*}		
		if $c \neq 0$.
	Because $Y \notin \dfrac{1}{d(V r)}  (\tau \mathbb{Z} +\mathbb{Z})$,
		the two lattices
		do not intersect, i.e.,
		\begin{equation*}
			\bigl(r Y + \frac{1}{d(r)} (\tau \mathbb{Z} + \mathbb{Z}) \bigr)
			\cap
			\bigl(- \frac{d}{c} Y + \frac{1}{c} (\tau \mathbb{Z} + \mathbb{Z}) \bigr)
			=
			\phi.
		\end{equation*}
	Therefore, we only have to prove the following two claims for the entireness of $LR(X) = L(X) - R(X)$:\\
	$\mathrm{(i)}$
		$L(X)$ is holomorphic at  $X  \in r Y + \dfrac{1}{d(r)} (\tau \mathbb{Z} + \mathbb{Z})$.\\
	$\mathrm{(ii)}$
		$L(X) - R(X)$ is holomorphic at $X \in - \dfrac{d}{c} Y + \dfrac{1}{c} (\tau \mathbb{Z} + \mathbb{Z})$ if $c > 0$.\\

	Firstly we prove $\mathrm{(i)}$.
	Let $z \in \dfrac{1}{d(r)} (\tau \mathbb{Z} + \mathbb{Z})$.
	By virtue of Lemma \ref{6.Lemma1},
		there are 
		$\bigl(\begin{smallmatrix} m'  &  n' \\m  &  n     \end{smallmatrix} \bigr),
			\bigl(\begin{smallmatrix} m_0'  &  n_0' \\ m_0  &  n_0   \end{smallmatrix} \bigr) \in \mathrm{M}_2(\mathbb{Z})$ 
		such that
		$\bigl( \begin{smallmatrix} m'  &  n' \\ m  &  n    \end{smallmatrix} \bigr) 
		=  \bigl( \begin{smallmatrix} m_0'  &  n_0' \\ m_0 &  n_0  \end{smallmatrix} \bigr)V$
		and
		\begin{equation*}
			\begin{split}
			d(r) z 
			& = n(r) (\tau m' + m) + d(r) (\tau n' + n) \\
			& = n(Vr) (\tau m_0' + m_0) + d(Vr) (\tau n_0' +n_0).
			\end{split}
		\end{equation*}
	Let $\vec{m}, \vec{n}, \vec{m_0}$ and $\vec{n_0}$ mean
		the vectors $(m',m), (n',n), (m_0', m_0)$ and $(n_0', n_0)$ respectively,
		that is, 
		\begin{equation}\label{7.LemProof11}
			\mbox{}^{t} \begin{pmatrix} \vec{m_0} \\ \vec{n_0} \end{pmatrix} V
			= 
			\mbox{}^{t} \begin{pmatrix} \vec{m} \\ \vec{n} \end{pmatrix} 
		\end{equation}
		where the symbol $\mbox{}^{t} \gamma$ means the transpose of any matrix $\gamma$.
	Then we set
		\begin{equation}\label{6.prop.aa}
			w  
			 : =
			-\vec{m}\cdot \vec{x} -\vec{n}\cdot \vec{y} = -\vec{m_0}\cdot \vec{x_0} -\vec{n_0}\cdot \vec{y_0}
				\in \mathbb{C}
		\end{equation}
		where the second equation follows from (\ref{7.eqn_MXY}) and (\ref{7.LemProof11}).
	Let $O$ stand for Landau notation.
	We see that
		\begin{eqnarray*}
			&  & \widehat{S}_{1,l}(r, M;  \begin{smallmatrix} X+r Y + z \\ Y \end{smallmatrix} ; \tau) 
			\\
			& \stackrel{\text{(\ref{8.def_S})}}{=}  &
			\frac{1}{d(r)}
				\sum_{j', j (  d(r))}
					\F \Bigl(\frac{\vec{j} + \vec{y}}{d(r)}; -d(r) (X + z);  \tau \Bigr) 
					\F^{(l-1)} \Bigl( n(r) \frac{\vec{j}  + \vec{y}}{d(r)}- \vec{x} ;-Y; \tau \Bigr) 
			\\
			& \stackrel{\text{(\ref{2.key_formula2})}}{=}  &
			\frac{(-1)^{l}}{d(r)}
				\sum_{j', j (  d(r))}
					\F \Bigl(\frac{\vec{j} + \vec{y}}{d(r)}; -d(r) X - n(r) (\tau m' + m) -  d(r) (\tau n'  + n);  \tau \Bigr)
			\\
			& & \qquad \qquad \qquad \qquad \times
					\F^{(l-1)} \Bigl(  - n(r) \frac{\vec{j}  + \vec{y}}{d(r)} + \vec{x} ;Y; \tau \Bigr) \\
			& \stackrel{\text{(\ref{2.key_formula1})}}{=}  &
			\frac{(-1)^{l}}{d(r)}
				\sum_{j', j (   d(r))} \e \bigl( - \frac{\vec{j} + \vec{y}}{d(r)} \cdot (n(r) \vec{m} + d(r) \vec{n}) \bigr)
			\\
			& & \qquad \qquad \qquad \qquad \times
					\F \Bigl(\frac{\vec{j} + \vec{y}}{d(r)}; -d(r) X ;  \tau \Bigr)
					\F^{(l-1)} \Bigl(  - n(r) \frac{\vec{j}  + \vec{y}}{d(r)} + \vec{x};Y; \tau \Bigr) 
			\\
			& \stackrel{\text{(\ref{3.definition_B})}}{=}  &
			\frac{(-1)^{l-1}\e(w)}{d(r)^2 X} 
				\sum_{j', j (  d(r))}
				\e(\vec{m}\cdot  \frac{-n(r) \vec{j} + d(r) \vec{x} - n(r) \vec{y} }{d(r)} )
			\\
			& & \qquad \qquad \qquad \qquad \times
				\F^{(l-1)} \Bigl( - n(r) \frac{\vec{j}  + \vec{y}}{d(r)} + \vec{x};Y; \tau \Bigr)
				+  O (1).
		\end{eqnarray*}
	Because of $\gcd (n(r), d(r)) = 1$ and (\ref{6.lemma2a}),  
			\begin{eqnarray*}
			&  & 	\widehat{S}_{1,l}(r, M;  \begin{smallmatrix} X+r Y + z \\ Y \end{smallmatrix} ; \tau)
			\\
			& = &
			\frac{(-1)^{l-1}\e(w)}{d(r)^2 X} 
				\sum_{j', j (  d(r))}
				\e(\vec{m}\cdot  \frac{ \vec{j} + d(r) \vec{x} - n(r) \vec{y} }{d(r)} )
			\\
			& & \qquad \qquad \qquad \qquad \times
				\F^{(l-1)} \Bigl( \frac{\vec{j}  + d(r) \vec{x} - n(r) \vec{y} }{d(r)};Y; \tau \Bigr)
				+  O (1)
			\\
			& = &
			\frac{(-1)^{l-1}\e(w)}{d(r)^l X} 
				\F^{(l-1)} \Bigl(  d(r) \vec{x} - n(r) \vec{y} ; \frac{1}{d(r)} Y+ \frac{\tau m' + m}{d(r)}; \tau \Bigr) +  O (1).
		\end{eqnarray*}
	In a similar way, 
 		we find that
		\begin{eqnarray*}
			& & \widehat{S}_{1,l}(V r, V M;  V \bigl( \begin{smallmatrix} X  + r Y + z\ \\ Y \end{smallmatrix} \bigr) ; \tau)
			\\
			&=&
			\frac{(-1)^{l}}{d(V r)} \sum_{j', j (  d(Vr))}
				\F \Bigl(\frac{\vec{j} + \vec{y_0}}{d(Vr)}; -d(r) X - n(Vr) (\tau m_0' + m_0) - d(Vr) (\tau n_0' +n_0);  \tau \Bigr) 
			\\
			& & \qquad \qquad \qquad \qquad  \times
				\F^{(l-1)} \Bigl(- n(V r) \frac{\vec{j} + \vec{y_0}}{d(V r)} + \vec{x_0}; c(X+z) + \frac{d(Vr)}{d(r)} Y ; \tau \Bigr)
			\\
			& = &
			\frac{(-1)^{l-1}\e(w)}{d(r)d(Vr) X} 
				\sum_{j', j (  d(r))}
				\e(\vec{m_0}\cdot  \frac{-n(Vr) \vec{j} + d(Vr) \vec{x_0} - n(Vr) \vec{y_0} }{d(Vr)} )
			\\
			& & \qquad \qquad \qquad \qquad \times
				\F^{(l-1)} \Bigl( - n(Vr) \frac{\vec{j}  + \vec{y_0}}{d(Vr)} + \vec{x_0};cz + \frac{d(Vr)}{d(r)} Y; \tau \Bigr)
				+  O (1).
			\\	
			&=&
			\frac{(-1)^{l-1}\e(w)}{d(r) d(V r)^{l-1} X} 
			\\
			& & \quad \times
				\F^{(l-1)} \Bigl(  d(V r) \vec{x_0} - n(Vr) \vec{y_0} ;  \frac{1}{d(r)} Y + \frac{cz + \tau m_0' + m_0}{d(Vr)}; \tau \Bigr) +  O (1).
		\end{eqnarray*}
	By (\ref{6.proof_Lemma4.2bb}) and (\ref{7.eqn_MXY}) we have $d(V r) \vec{x_0} - n(Vr) \vec{y_0} = d(r) \vec{x} - n(r) \vec{y}$,
	and by some calculations we get
		\begin{equation*}
			j(V;r)^{l-1}
			=
			\frac{d(V r)^{l-1}}{d(r)^{l-1}},
			\qquad
			\frac{\tau m' +m}{d(r)}
			=
			\frac{cz + \tau m_0' +m_0}{d(Vr)}.
		\end{equation*}
	Therefore
		it follows that
		$L(X  + r Y + z) = O(1)$, 
		which prove the claim $\mathrm{(i)}$.
	
	We prove the claim $\mathrm{(ii)}$ next. 
	Let $z = \dfrac{1}{c}(\tau m_0' +m_0)  \in \dfrac{1}{c} (\tau \mathbb{Z} + \mathbb{Z})$.	
	By (\ref{3.definition_B}), one has
		\begin{equation}\label{6.proposition1ss}
			\begin{split}
			\F^{(n)} (\vec{x};X;\tau) 
				& =
				\frac{(-1)^n n!}{(2\pi i)^n X^{n+1}} + 	
				\sum_{m=0}^{\infty}
				\frac{B_{m+n+1}(\vec{x};\tau)}{(m+n+1)m!} (2\pi i)^{m+1} X^m \\
				& =
				\frac{(-1)^n n!}{(2\pi i)^n X^{n+1}}  + O(1), 
			\end{split}
		\end{equation}
		thus we find that
		\begin{eqnarray*}
			& & \widehat{S}_{1,l}(V r, V M;  V \bigl( \begin{smallmatrix} X  - (d/c) Y + z\ \\ Y \end{smallmatrix} \bigr) ; \tau)\\
			&\stackrel{\text{(\ref{2.key_formula2})}}{\stackrel{\text{(\ref{6.proposition1cc})}}{=}} &
			\frac{(-1)^{l-1}}{d(V r)} \sum_{j', j (  d(Vr))}
					\F \Bigl(- \frac{\vec{j} + \vec{y_0}}{d(Vr)}; d(r) (X + z) - \frac{d(Vr) Y}{c} ;  \tau \Bigr) \\
			& & \qquad \qquad \qquad \qquad  \times
					\F^{(l-1)} \Bigl(- n(V r) \frac{\vec{j} + \vec{y_0}}{d(V r)} + \vec{x_0}; cX + \tau m_0' + m_0 ; \tau \Bigr)	\\
			&\stackrel{\text{(\ref{2.key_formula1})}}{=} &
			\frac{(-1)^{l-1}}{d(V r)} \sum_{j', j (  d(Vr))}
					\e \bigl(-   \vec{m_0} \cdot \{ n(V r) \frac{\vec{j} + \vec{y_0}}{d(V r)} - \vec{x_0}  \} \bigr)
			\\ 
			& &    \times
					\F \Bigl(- \frac{\vec{j} + \vec{y_0}}{d(Vr)}; d(r) (X + z) - \frac{d(Vr) Y}{c} ;  \tau \Bigr)  
					\F^{(l-1)} \Bigl(- n(V r) \frac{\vec{j} + \vec{y_0}}{d(V r)} + \vec{x_0}; cX; \tau \Bigr)	\\	
			&\stackrel{\text{(\ref{6.proposition1ss})}}{=} &
			\frac{(l-1)!}{(2 \pi i)^{l-1}c^l}
				\e( \vec{m_0} \cdot \vec{x_0}) 
				\sum_{k=0}^{l-1} \frac{( 2 \pi i d(r))^k}{k ! X^{l-k}} \frac{1}{d(V r) } 
			\\
			& &    \times	
					 \sum_{j', j (  d(Vr))} 
						\e (-  n(V r)  \vec{m_0} \cdot \frac{\vec{j} + \vec{y_0}}{d(V r)} ) 
						\F^{(k)} \Bigl(- \frac{\vec{j} + \vec{y_0}}{d(Vr)};  d(r) z - \frac{d(V r) Y}{c};  \tau \Bigr) + O(1)\\
			&\stackrel{\text{(\ref{6.lemma2a})}}{=}&
			\frac{(l-1)!}{(2 \pi i)^{l-1}c^l}
				\e( \vec{m_0} \cdot \vec{x_0}) 
			\\
			& &  \times
				\sum_{k=0}^{l-1} \frac{( 2 \pi i d(r))^k}{k ! X^{l-k} d(V r)^{k}} 
						\F^{(k)} \Bigl(- \vec{y_0};  - \frac{Y}{c} +   \frac{d(r) + c n(Vr)}{d(Vr)}  z ;  \tau \Bigr) + O(1).		
		\end{eqnarray*}
	Because of $d(r)^k/d(Vr)^k = 1/j(V;r)^k$ and $(d(r) + c n(Vr))/d(Vr) = a$,  we conclude that
		\begin{multline}\label{6.proposition1ee}
			L( X  - \frac{d}{c} Y + z)\\
			=
			- \frac{(l-1)!}{(2 \pi i)^{l-1}c^l}
				\e( \vec{m_0} \cdot \vec{x_0}) 
				\sum_{k=0}^{l-1} \frac{( 2 \pi i)^k j(V; r)^{l-1-k}}{k ! X^{l-k}} 
						\F^{(k)} \Bigl(- \vec{y_0};  - \frac{Y}{c} + a z ;  \tau \Bigr) + O(1).			
		\end{multline}
	On the other hand, 
			since 
			$\bigl( \begin{smallmatrix} \vec{x} \\ \vec{y}\end{smallmatrix} \bigr)
				=
				V^{-1} \bigl( \begin{smallmatrix} \vec{x_0} \\ \vec{y_0}\end{smallmatrix} \bigr)$,
			it follows that $\vec{y} = -c\vec{x_0} + a \vec{y_0}$ or
		\begin{equation}\label{6.proposition1fff}
			\vec{m_0} \cdot \vec{y} = a \vec{m_0} \cdot \vec{y_0} -c \vec{m_0}  \cdot \vec{x_0}.
		\end{equation}
	Therefore we can find similarly above way that 
		\begin{eqnarray*}
			& &
			R( X  - \frac{d}{c} Y + z)
			\\
			& = &
			\widehat{R}_V (l, r, M; \begin{smallmatrix} X  - (d/c) Y + z  \\ Y \end{smallmatrix} ; \tau)
			\\
			& \stackrel{\text{(\ref{6.definition_R})}}{=} &			
			(-1)^l \sum_{k=0}^{l-1}\binom{l-1}{k} 
				\bigl( - j (V, r)\bigr)^k
					\frac{1}{c} 
			\\
			& & \qquad \qquad \times
			\sum_{j', j (  \abs{c})} 
						\F^{(k)} (\frac{\vec{j} +\vec{y}}{c};	 -c(X+z) ;\tau)
						\F^{(l-1-k)} (d\frac{\vec{j} +\vec{y}}{c} + \vec{x}; Y ;\tau) 
			\\ 
			& \stackrel{\text{(\ref{8.def_S})}}{\stackrel{\text{(\ref{6.proposition1ss})}}{=}}  &
			- \sum_{k=0}^{l-1}
				\binom{l-1}{k} j (V, r)^k  \frac{k!}{(2 \pi i)^k (cX)^{k+1}}
				\frac{1}{c}  
			\\
			& & \qquad \qquad \times
			\sum_{j', j (  \abs{c})}
					\e(- \vec{m_0} \cdot \frac{\vec{j} +\vec{y}}{c})
						\F^{(l-1-k)} (\frac{- d \vec{j} - \vec{y_0} }{c} ; -Y ;\tau) + O(1)
			\\
			& \stackrel{\text{(\ref{6.proposition1fff})}}{=} &
			- \sum_{k=0}^{l-1}
				\binom{l-1}{k} j (V, r)^k  \frac{k! 
					\e (\vec{m_0}  \cdot \vec{x_0})}{(2 \pi i)^k (cX)^{k+1}}
				\frac{1}{c}  
			\\
			& & \qquad \qquad \times
			\sum_{j', j (  \abs{c})}
					\e( - a \vec{m_0} \cdot \{  \frac{ d \vec{j} +  \vec{y_0} }{c} \} )
						\F^{(l-1-k)} (\frac{- d \vec{j} - \vec{y_0} }{c} ; -Y ;\tau) + O(1)\\			
			& \stackrel{\text{(\ref{6.lemma2a})}}{=} &
			- \frac{\e (\vec{m_0}  \cdot \vec{x_0})}{c^{l}}\sum_{k=0}^{l-1}
				\binom{l-1}{k} j (V, r)^k  \frac{k! }{(2 \pi i)^k X^{k+1}}
			\\
			& & \qquad \qquad \qquad \qquad \times
						\F^{(l-1-k)} (-\vec{y_0}; -\frac{Y}{c} +  a z;\tau) + O(1).							
		\end{eqnarray*}
	In the last equation, we use the fact that $ad \equiv 1 \  (\mathrm{mod\ } \abs{c})$
		which follows from $\det V =1$.
		Replacing $k$ by $l-1-k$, the last formula equals the right hand side of (\ref{6.proposition1ee}),
			which proves the claim $\mathrm{(ii)}$.
\end{proof}

The remain of the tasks in this section is to derive THEOREM \ref{6.Theorem1} from PROPOSITION \ref{6.Proposition1}.
In order to do it, we need the following lemma:
\begin{lemma}\label{6.Lemma3}
	Let $M(X)$ means the set of meromorphic functions with the variable $X$,
		and $MF(Y, Z)$ that with the variables $X, Y$.
	For $f(X) \in MF(X)$ and $g(X,Y) \in MF(X,Y)$, 
		let $C_{X^m} (f(X))$, $C_{X^m} (g(X,Y))$ and $C_{Y^m} (g(X,Y))$ mean
		the coefficients of $X^m$ in $f(X)$ and $g(X,Y)$, and that of $Y^m$ in $g(X,Y)$  respectively.
	We note that 
		$C_{X^m} (f(X)) \in \mathbb{C}$, 
		$C_{X^m} (g(X,Y)) \in MF(Y)$ 
		and 
		$C_{Y^m} (g(X,Y)) \in MF(X)$.
	
	Let $l, r, V, M$ be as in PROPOSITION \ref{6.Proposition1} with extra condition $c \neq 0$.
	Then we have the three equations
		\begin{equation}\label{6.lemma3a}
			C_{X^0} \circ C_{Y^0} (\widehat{S}_{1,l}(r, M;  \begin{smallmatrix} X \\ Y \end{smallmatrix} ; \tau))
			=
			(2 \pi i)^2
			\Bigl[
				\frac{1}{l} S_{1,l}(r, M ; \tau) - \frac{r^l}{l(l+1)} B_{l+1} ( \vec{y};  \tau ) 
			\Bigr],
		\end{equation}
		\begin{multline}\label{6.lemma3b}
			C_{X^0} \circ C_{Y^0} (\widehat{S}_{1,l}(V r, V M;  V \bigl( \begin{smallmatrix} X \\ Y \end{smallmatrix} \bigr) ; \tau))
			\\
			=
			(2 \pi i)^2
			\Bigl[
				\frac{1}{l} S_{1,l}(V r, V M; \tau)
				+
				\frac{(-1)^{l-1}}{l (l+1) c^l j(V;r)^l} B_{l+1} (c\vec{x}+d\vec{y}; \tau)
				\\
				+
				\frac{c}{(l+1) d(r) d(Vr)^l} B_{l+1}(n(r) \vec{y} - d(r) \vec{x};\tau)		
			\Bigr],
		\end{multline}
		\begin{multline}\label{6.lemma3c}
			C_{X^0} \circ C_{Y^0} (\widehat{R}_V (l, r, M; \begin{smallmatrix} X  \\ Y \end{smallmatrix} ; \tau))
			\\
			=
			\frac{(2 \pi i)^2}{l(l+1)} 
			\Bigl[
				\frac{(-1)^l (l+1)!}{(2\pi i)^{l+1}}
				R_V (l, r, M; \tau)
				+
				\frac{(-1)^l}{c^l j(V;r)} B_{l+1} (c\vec{x}+d\vec{y};\tau)
				-
				r^lB_{l+1} (\vec{y}; \tau) \
			\Bigr].
		\end{multline}
\end{lemma}

\begin{proof}
	We find from (\ref{8.def_S}), (\ref{6.lemma2a}) and (\ref{6.proposition1ss}) that
		\begin{eqnarray*}
			& & C_{Y^0}(\widehat{S}_{1,l}(r; M,  \begin{smallmatrix} X \\ Y \end{smallmatrix} ; \tau) ) 
			\\
			&\stackrel{\text{(\ref{8.def_S})}}{\stackrel{\text{(\ref{6.proposition1ss})}}{=}} &
			\frac{2 \pi i}{l d(r)}
				\sum_{j', j (  d(r))} 
					\Bigl{\{}
						\F \Bigl(\frac{\vec{j} + \vec{y}}{d(r)} ;  -d(r) X;  \tau \Bigr) B_l \Bigl(n(r) \frac{\vec{j} + \vec{y}}{d(r)} -\vec{x}; \tau \Bigr)
			\\
			& & \qquad \qquad \qquad \qquad
						-
						(n(r))^l \F^{(l)} \Bigl(\frac{\vec{j} + \vec{y}}{d(r)} ;  -d(r) X;  \tau \Bigr) 
					\Bigr{\}} 
			\\
			& \stackrel{\text{(\ref{6.lemma2a})}}{=} &
			\frac{2 \pi i}{l}
				\biggl[ \frac{1}{d(r)}
					\sum_{j', j (  d(r))} \F \Bigl(\frac{\vec{j} + \vec{y}}{d(r)} ;  -d(r) X;  \tau \Bigr) 
						B_l \Bigl(n(r) \frac{\vec{j} + \vec{y}}{d(r)} -\vec{x}; \tau \Bigr)
			\\
			& & \qquad \qquad \qquad \qquad
						- r^l \F^{(l)} ( \vec{y} ;  - X;  \tau )
				\biggr].
		\end{eqnarray*}
	By using (\ref{6.proposition1ss}) again, we get (\ref{6.lemma3a}).
	
	We prove (\ref{6.lemma3b}) next.
	By (\ref{6.proposition1cc}), we have
		\begin{multline*}
			C_{Y^0} (\widehat{S}_{1,l}(V r, V M;  V \bigl( \begin{smallmatrix} X \\ Y \end{smallmatrix} \bigr) ; \tau))
			\\
			=
			\frac{1}{d(V r)} \sum_{j', j (  d(Vr))}
				\F \Bigl(\frac{\vec{j} + \vec{y_0}}{d(Vr)}; -d(r) X;  \tau \Bigr)
				\F^{(l-1)} \Bigl(n(V r) \frac{\vec{j} + \vec{y_0}}{d(V r)} -\vec{x_0};-cX; \tau \Bigr).
		\end{multline*}
	By (\ref{6.proposition1ss}), we get (\ref{6.lemma3b}).
	
	Finally we verify (\ref{6.lemma3c}).
	By (\ref{6.definition_R}), (\ref{6.lemma2a}) and (\ref{6.proposition1ss}) we have
		\begin{eqnarray*}
			& & 
			C_{Y^0}(\widehat{R}_V (l, r, M; \begin{smallmatrix} X  \\ Y \end{smallmatrix} ; \tau)) 
			\\
			&\stackrel{\text{(\ref{6.definition_R})}}{\stackrel{\text{(\ref{6.proposition1ss})}}{=}} &
			2 \pi i (-1)^l \sum_{k=0}^{l-1}
				\binom{l-1}{k} \bigl(- j (V, r)\bigr)^k  \frac{1}{l-k} \\
			& & \qquad \qquad \times
				 \frac{1}{c} \sum_{j', j (  \abs{c})} 
					\Bigl{\{}
						\F^{(k)} \Bigl(  \frac{\vec{j} + \vec{y}}{c} ;  -c X;  \tau \Bigr)
							B_{l-k} \Bigl( d \frac{\vec{j} + \vec{y}}{c} + \vec{x}; \tau \Bigr)
			\\
			& & \qquad \qquad \qquad \qquad
						-
						d^{l-k}
							\F^{(l)} \Bigl(  \frac{\vec{j} + \vec{y}}{c} ;  -c X;  \tau \Bigr)	
					\Bigr{\}} 
			\\
			& \stackrel{\text{(\ref{6.lemma2a})}}{=} &
			2 \pi i (-1)^l \sum_{k=0}^{l-1}
				\binom{l-1}{k} \bigl(- j (V, r)\bigr)^k \frac{1}{l-k}\\
			& & \qquad \qquad \times
				\Bigl{\{}
					 \frac{1}{c} \sum_{j', j (  \abs{c})} 
						\F^{(k)} \Bigl(  \frac{\vec{j} + \vec{y}}{c} ;  -c X;  \tau \Bigr)
							B_{l-k} \Bigl( d \frac{\vec{j} + \vec{y}}{c} + \vec{x}; \tau \Bigr) 
			\\
			& & \qquad \qquad \qquad \qquad
						-
						\frac{d^{l-k}}{c^l}
							\F^{(l)} \Bigl( \vec{y} ;  - X;  \tau \Bigr)	
				\Bigr{\}} 
			\\
			& =&
			2 \pi i (-1)^l 
				\Bigl[
					\sum_{k=0}^{l-1} \binom{l-1}{k} \frac{\bigl(- j (V, r)\bigr)^k}{l-k}
						 \frac{1}{c} \sum_{j', j (  \abs{c})} 
							\F^{(k)} \Bigl(  \frac{\vec{j} + \vec{y}}{c} ;  -c X;  \tau \Bigr)
								B_{l-k} \Bigl( d \frac{\vec{j} + \vec{y}}{c} + \vec{x}; \tau \Bigr) \\
			& & \qquad \qquad \qquad \qquad 
					-
					\frac{1}{c^l} \F^{(l)} \Bigl( \vec{y} ;  - X;  \tau \Bigr)
						\sum_{k=0}^{l-1} \binom{l-1}{k} \frac{\bigl(- j (V, r)\bigr)^k d^{l-k}}{l-k}	
				\Bigr]	.
		\end{eqnarray*}
	Since 
		\begin{eqnarray*}
			\sum_{k=0}^{l-1} \binom{l-1}{k} \frac{\bigl(- j (V, r)\bigr)^k d^{l-k}}{l-k}	
			& = &
			\frac{1}{l} \sum_{k=0}^{l-1} \binom{l}{k} \bigl(- j (V, r)\bigr)^k d^{l-k}
			\\
			& = &
			\frac{1}{l} \Bigl(
						(d - j(V, r))^{l}
						-
						(- j(V, r))^{l}
					\Bigr)		
			\\
			& = &
			\frac{(-1)^l}{l} 
				\Bigl( (cr)^l - j(V;r)^l \Bigr),
		\end{eqnarray*}
		we get
		\begin{eqnarray*}
			& &
			C_{Y^0}(\widehat{R}_V (l, r, M; \begin{smallmatrix} X  \\ Y \end{smallmatrix} ; \tau))
			\\
			& = &
			\frac{2 \pi i}{l}
				\Bigl[
					\frac{j(V;r)^l - (cr)^l}{c^l} \F^{(l)} \Bigl( \vec{y} ;  - X;  \tau \Bigr)
			\\
			& & \quad
					+
					(-1)^l \sum_{k=0}^{l-1} \binom{l}{k} \bigl(- j (V, r)\bigr)^k
						\frac{1}{c} \sum_{j', j (  \abs{c})} 
							\F^{(k)} \Bigl(  \frac{\vec{j} + \vec{y}}{c} ;  -c X;  \tau \Bigr)
								B_{l-k} \Bigl( d \frac{\vec{j} + \vec{y}}{c} + \vec{x}; \tau \Bigr) 
				\Bigl].
		\end{eqnarray*}
	Therefor it follows from (\ref{6.proposition1ss}) that
		\begin{eqnarray*}
			& &
			C_{X^0} \circ C_{Y^0} (\widehat{R}_V (l, r, M; \begin{smallmatrix} X  \\ Y \end{smallmatrix} ; \tau))
			\\
			& =&
			\frac{(2 \pi i)^2}{l (l+1)}
					\Bigl[
						\frac{j(V;r)^l - (cr)^l}{c^l} B_{l+1} ( \vec{y} ;  \tau )
				\\
			& & \quad
						+
						(-1)^l \sum_{k=0}^{l-1} \binom{l+1}{k+1} \bigl(- j (V, r)\bigr)^k
							\frac{1}{c} \sum_{j', j (  \abs{c})} 
								B_{k+1} \Bigl(  \frac{\vec{j} + \vec{y}}{c} ;  \tau \Bigr)
									B_{l-k} \Bigl( d \frac{\vec{j} + \vec{y}}{c} + \vec{x}; \tau \Bigr) 
					\Bigl]
			\\
			& =&
			\frac{(2 \pi i)^2}{l (l+1)}
					\Bigl[
						\frac{j(V;r)^l - (cr)^l}{c^l} B_{l+1} ( \vec{y} ;  \tau )
						+
						\frac{(-1)^l (l+1)!}{(2\pi i)^{l+1}}R_V (l, r, M; \tau)
				\\
			& & \quad
						+
						\frac{(-1)^l}{j(V;r) c} \sum_{j', j (  \abs{c})} B_{l+1} \Bigl( d \frac{\vec{j} + \vec{y}}{c} + \vec{x}; \tau \Bigr)
						-
						\frac{j(V;r)^l}{c}  \sum_{j', j (  \abs{c})} B_{l+1} \Bigl(  \frac{\vec{j} + \vec{y}}{c} ;  \tau \Bigr)
					\Bigl]
		\\
		& =&
			\frac{(2 \pi i)^2}{l (l+1)}
			\Bigl[
				\frac{j(V;r)^l - (cr)^l}{c^l} B_{l+1} ( \vec{y} ;  \tau )
				+
				\frac{(-1)^l (l+1)!}{(2\pi i)^{l+1}}R_V (l, r, M; \tau)
				\\
				& & \quad
				+
				\frac{(-1)^l}{j(V;r) c^l} B_{l+1} (c\vec{x}+d\vec{y}; \tau )
				-
				\frac{j(V;r)^l}{c^l} B_{l+1} ( \vec{y};  \tau )
			\Bigl],
		\end{eqnarray*}
	which with (\ref{6.lemma2b}) gives (\ref{6.lemma3c}).
	
\end{proof}
We prove Theorem \ref{6.Theorem1} as the end of the section.
\begin{proof}[Proof of Theorem \ref{6.Theorem1}]
	The case that $c =0$ is trivial, so we suppose that $c \neq 0$.
	If $j(V;r) > 0$, (\ref{6.theorem1a}) follows from Proposition \ref{6.Proposition1} and Lemma \ref{6.Lemma3}.
	Since $j(-V; r) = -j(V; r)$, $S_{1,l}((-V) r, (-V) M; \tau) = (-1)^{l-1} S_{1,l}(V r, V M; \tau)$
		and $R_{-V} (l, r, M; \tau)= R_V (l, r, M; \tau)$ by definition, 
		the case that $j(V;r) <0$ is reduced to the case that $j(V;r) >0$.
\end{proof}



\begin{flushleft}
\mbox{}\\ \qquad
MACHIDE, Tomoya
\mbox{}\\ \qquad
Research Center for Quantum Computing
\mbox{}\\ \qquad
Interdisciplinary Graduate School of Science and Engineering
\mbox{}\\ \qquad
Kinki University
\mbox{}\\ \qquad
3-4-1 Kowakae, Higashi-Osaka, Osaka 577-8502, Japan
\mbox{}\\ \qquad
E-mail: machide.t@gmail.com
\end{flushleft}

\end{document}